\documentclass[12pt,reqno]{amsart}%
\usepackage{amsmath, amsfonts, amssymb, amsthm, amscd, amsbsy}
\usepackage{fancyhdr}
\usepackage[usenames,dvipsnames,svgnames,x11names,hyperref]{xcolor}
\usepackage{geometry}
\usepackage{graphicx}
\usepackage{hyperref}
\usepackage{amsmath}
\usepackage{amsfonts}
\usepackage{amssymb}
\usepackage{lineno}%
\providecommand{\U}[1]{\protect\rule{.1in}{.1in}}
\hypersetup{colorlinks,breaklinks,
	linkcolor={Fuchsia},
	citecolor={ForestGreen},
	urlcolor={NavyBlue}}
\newtheorem{theorem}{Theorem}[section]
\theoremstyle{plain}

\newtheorem{lemma}{Lemma}[section]

\numberwithin{equation}{section}
\allowdisplaybreaks
\newtheorem{theorema}{Theorem}[section]

\begin{document}
\title[Stability of second order HUP]{Sharp stability of the Heisenberg Uncertainty Principle: Second-Order and
Curl-Free Field Cases}
\author{Anh Xuan Do}
\address{Anh Xuan Do: Department of Mathematics, University of Connecticut, Storrs, CT
06269, USA}
\email{anh.do@uconn.edu}
\author{Nguyen Lam}
\address{Nguyen Lam: School of Science and the Environment, Grenfell Campus, Memorial
University of Newfoundland, Corner Brook, NL A2H5G4, Canada}
\email{nlam@mun.ca}
\author{Guozhen Lu}
\address{Guozhen Lu: Department of Mathematics, University of Connecticut, Storrs, CT
06269, USA}
\email{guozhen.lu@uconn.edu}

\begin{abstract}
Using techniques from harmonic analysis, we derive several sharp stability
estimates for the second order Heisenberg Uncertainty Principle. We also
present the explicit lower and upper bounds for the sharp stability constants
and compute their exact limits when the dimension $N\rightarrow\infty$. Our
proofs rely on spherical harmonics decomposition and Fourier analysis,
differing significantly from existing approaches in the literature. Our
results substantially improve the stability constants of the second order
Heisenberg Uncertainty Principle recently obtained in \cite{DN25}. As direct
consequences of our main results, we also establish the sharp stability, with exact
asymptotic behavior of the stability constants, of the Heisenberg Uncertainty
Principle with curl-free vector fields and a sharp version of the second order
Poincar\'{e} type inequality with Gaussian measure.

\end{abstract}
\maketitle

\section{Introduction}

Ever since Brezis and Lieb raised a question about the stability of the
Sobolev inequality in \cite{BL85}, the studies of quantitative stability for
functional and geometric inequalities have attracted great attention. Here the
stability of functional and geometric inequalities refers to the study of how
much a function or a geometric object deviates from being optimal when the
inequality is "almost" satisfied. Mathematically, suppose that we have the
inequality $A\left(  u\right)  \geq B\left(  u\right)  $ for all $u$ belonging
to a suitable space, and that $A\left(  u\right)  =B\left(  u\right)  $ if and
only if $u\in M$, the set of all optimizers, then we would like to find a
constant $c_{S}>0$ such that the following stability inequality holds:
$\delta\left(  u\right)  \geq c_{S}d\left(  u,M\right)  $. Here $\delta\left(
u\right)  :=A\left(  u\right)  -B\left(  u\right)  $ is the (nonnegative)
deficit function, and $d\left(  u,M\right)  $ is a suitable distance from $u$
to the set $M$. The literature on the topic is extremely vast and therefore,
we just refer the interested reader to \cite{BDNS20, CF13, CFW13, Chr14,
Chr17, Cian09, CL13, DT16, FJ1, FJ2, FMP10, FN19, FZ22, FMP08, Kon25}, to name
just a few.

It is worthy to note that the stability constants $c_{S}$ and whether or not
the stability inequalities can be attained have not been investigated in the
literature until very recently. For instance, the question of Brezis and Lieb
was answered affirmatively by Bianchi and Egnell in \cite{BE91} that
\[
\int_{\mathbb{R}^{N}}\left\vert \nabla u\right\vert ^{2}dx-S_{N}\left(
\int_{\mathbb{R}^{N}}|u|^{\frac{2N}{N-2}}dx\right)  ^{\frac{N-2}{N}}\geq
c_{S}\inf_{U\in E_{Sob}}\int_{\mathbb{R}^{N}}\left\vert \nabla\left(
u-U\right)  \right\vert ^{2}dx
\]
for some stability constant $c_{S}>0$. Here $S_{N}$ is the sharp Sobolev
constant and $E_{Sob}$ is the manifold of the optimizers of the Sobolev
inequality. However, the precise information on the stability constant $c_{S}$
was totally missing in the literature until the very recent results in
\cite{DEFFL}. Indeed, in \cite{DEFFL}, Dolbeault, Esteban, Figalli, Frank and
Loss pioneered the study of the sharp stability Sobolev constant $c_{S}$ and
provided some optimal lower bounds for $c_{S}$ when the dimension
$N\rightarrow\infty$, and established the stability for Gaussian log-Sobolev
inequality as an application. Also, K\"{o}nig proved in \cite{Kon23} that the
optimal lower bound is strictly smaller than the spectral gap constant
$\frac{4}{N+4}$, and derived in \cite{Kon22} its attainability. In a recent work \cite{CLT24}, Chen, Tang and the third
author investigated explicit lower bounds for the stability of the
Hardy-Littlewood-Sobolev (HLS) inequalities. Their analysis further yielded
explicit stability estimates for the classical Sobolev inequalities, as well
as their higher-order and fractional counterparts. Additionally, in
\cite{CLT242, CLT243}, they established optimal asymptotic lower bounds for
the HLS and higher/fractional Sobolev inequalities in the regimes when
$N\rightarrow\infty$ for $0<s<N/2$ and when $s\rightarrow0$ for all $N$. The
latter case, where $s\rightarrow0$, enabled them to derive a global stability
result for the log-Sobolev inequality on the sphere, originally proved by
Beckner \cite{Beckner93}, thereby sharpening the earlier local stability
obtained by Chen et al. in \cite{CLT23}. More recently, Chen et al.~\cite{CLTW} established the optimal stability result for the Sobolev inequality on the Heisenberg group. Since both the P\'olya--Szeg\"o inequality and the Riesz rearrangement inequality fail in this setting, the gradient flow and integral flow methods used in \cite{DEFFL, CLT24, CLT242, CLT243} are no longer applicable. Consequently, the authors of \cite{CLTW} introduced a new approach based on the CR Yamabe flow to bridge local stability and global stability. It is also
worth noting that, in general, studying the stability of higher-order and fractional-order
functional and geometric inequalities can be more complex, and thus is usually less understood, due to several challenges and obstacles when dealing with these
situations such as the nonlocality, the lack of the symmetrization argument,
etc.

The main objective of this article is the stability of the well-known
Heisenberg Uncertainty Principle (HUP). It is worth mentioning that the HUP
plays a fundamental role in quantum mechanics and mathematical physics. For
instance, HUP has been used, both implicitly and explicitly, to provide
essential inequalities for proving the stability of matter (see the book of
Lieb and Seiringer \cite{LS10}), the stability of relativistic matter (Lieb
and Loss \cite{LL02}, Lieb and Yau \cite{LY88}), bounds in spectral theory
(Reed and Simon \cite{RS75}), and constraints in harmonic analysis (Fefferman
\cite{Fef83}). In mathematical terms, it can be stated as follows: For $u\in
S_{0}$
\begin{equation}
\left(  \int_{\mathbb{R}^{N}}\left\vert \nabla u\right\vert ^{2}dx\right)
\left(  \int_{\mathbb{R}^{N}}\left\vert x\right\vert ^{2}\left\vert
u\right\vert ^{2}dx\right)  \geq\frac{N^{2}}{4}\left(  \int_{\mathbb{R}^{N}%
}\left\vert u\right\vert ^{2}dx\right)  ^{2}. \label{HUP}%
\end{equation}
Here $S_{0}$ is the completion of $C_{0}^{\infty}\left(  \mathbb{R}%
^{N}\right)  $ under the norm $\left(  \int_{\mathbb{R}^{N}}\left\vert \nabla
u\right\vert ^{2}dx\right)  ^{\frac{1}{2}}+\left(  \int_{\mathbb{R}^{N}%
}\left\vert x\right\vert ^{2}\left\vert u\right\vert ^{2}dx\right)  ^{\frac
{1}{2}}$. The constant $\frac{N^{2}}{4}$ is optimal and can be attained by the
Gaussian signal $\alpha e^{-\beta\left\vert x\right\vert ^{2}}$, $\alpha
\in\mathbb{R}$, $\beta>0$.

In the same spirit of Brezis and Lieb's question, one could also ask whether
or not the HUP (\ref{HUP}) is stable. More clearly, is it true that
$\delta\left(  u\right)  \approx0$ implies $u\approx\alpha e^{-\beta\left\vert
x\right\vert ^{2}}$ in some sense for some $\alpha\in\mathbb{R}$, $\beta>0$,
and for some Heisenberg deficit $\delta$? This question was first addressed by
McCurdy and Venkatraman in \cite{MV21}. More precisely, the authors in
\cite{MV21} applied the concentration-compactness arguments and proved that
there exist universal constants $C_{1}>0$ and $C_{2}(N)>0$ such that
\[
\theta_{2}(u)\geq C_{1}\left(  \int_{\mathbb{R}^{N}}\left\vert u\right\vert
^{2}dx\right)  d_{1}^{2}(u,E_{HUP})+C_{2}(N)d_{1}^{4}(u,E_{HUP}),
\]
for all $u\in S_{0}$. Here
\[
\theta_{2}\left(  u\right)  :=\left(  \int_{\mathbb{R}^{N}}|\nabla
u|^{2}dx\right)  \left(  \int_{\mathbb{R}^{N}}\left\vert x\right\vert
^{2}\left\vert u\right\vert ^{2}dx\right)  -\dfrac{N^{2}}{4}\left(
\int_{\mathbb{R}^{N}}\left\vert u\right\vert ^{2}dx\right)  ^{2}%
\]
is a HUP deficit, $E_{HUP}:=\left\{  \alpha e^{-\beta\left\vert x\right\vert
^{2}}:\alpha\in\mathbb{R},\beta>0\right\}  $ is the set of all the optimizers
of (\ref{HUP}), and $d_{1}(u,A):=\inf_{v\in A}\left\{  \left\Vert
u-v\right\Vert _{2}\right\}  $ is the distance from $u$ to the set $A$.
Therefore, the results in \cite{MV21} imply that if $\theta_{2}\left(
u\right)  \approx0$, then $u\approx\alpha e^{-\beta\left\vert x\right\vert
^{2}}$ in $L^{2}\left(  \mathbb{R}^{N}\right)  $ for some $\alpha\in
\mathbb{R}$, $\beta>0$. However, because of the approach in \cite{MV21}, no
information on the constants $C_{1}$ and $C_{2}(N)$ were provided.

Shortly after the results in \cite{MV21}, a simple and constructive proof was
provided by Fathi in \cite{F21} to show that $C_{1}=\frac{1}{4}$ and
$C_{2}=\frac{1}{16}$. However, these constants are far from being sharp.
Eventually, the best constants $C_{1}$ and $C_{2}$ of the stability of the HUP
were established recently in \cite{CFLL24}. The strategy in \cite{CFLL24} is
as follows: Firstly, the authors changed the deficit function and worked on
the following one:
\[
\theta_{1}\left(  u\right)  :=\left(  \int_{\mathbb{R}^{N}}\left\vert \nabla
u\right\vert ^{2}dx\right)  ^{\frac{1}{2}}\left(  \int_{\mathbb{R}^{N}%
}\left\vert u\right\vert ^{2}\left\vert x\right\vert ^{2}dx\right)  ^{\frac
{1}{2}}-\frac{N}{2}\int_{\mathbb{R}^{N}}\left\vert u\right\vert ^{2}dx
\]
which turns out to be more suitable in studying the stability of HUP.
Secondly, they proved the following HUP identity that provides simple and
direct understandings to the HUP as well as its optimizers: For $u\in S_{0}$,
$u\neq0$ and $\lambda=\left(  \frac{\int_{\mathbb{R}^{N}}\left\vert
u\right\vert ^{2}\left\vert x\right\vert ^{2}dx}{\int_{\mathbb{R}^{N}}|\nabla
u|^{2}dx}\right)  ^{\frac{1}{4}}$, one has
\begin{equation}
\theta_{1}\left(  u\right)  =\frac{\lambda^{2}}{2}\int_{\mathbb{R}^{N}%
}\left\vert \nabla\left(  ue^{\frac{\left\vert x\right\vert ^{2}}{2\lambda
^{2}}}\right)  \right\vert ^{2}e^{-\frac{\left\vert x\right\vert ^{2}}%
{\lambda^{2}}}dx. \label{HUPI}%
\end{equation}
Obviously, it can also deduced from (\ref{HUPI}) that all optimizers for
(\ref{HUP}) are the classical Gaussian profiles. Finally, the authors in
\cite{CFLL24} combined the HUP identity (\ref{HUPI}) and the following optimal
Poincar\'{e} inequality with Gaussian type measure: for all $\lambda\neq0$,
\[
\int_{\mathbb{R}^{N}}|\nabla u|^{2}e^{-\frac{1}{2\left\vert \lambda\right\vert
^{2}}\left\vert x\right\vert ^{2}}dx\geq\frac{1}{\left\vert \lambda\right\vert
^{2}}\inf_{c}\int_{\mathbb{R}^{N}}\left\vert u-c\right\vert ^{2}e^{-\frac
{1}{2\left\vert \lambda\right\vert ^{2}}\left\vert x\right\vert ^{2}}dx\text{,
}%
\]
to show the following sharp stability of HUP:

\begin{theorema}
\label{A}For all $u\in S_{0}:$%
\begin{equation}
\left(  \int_{\mathbb{R}^{N}}\left\vert \nabla u\right\vert ^{2}dx\right)
^{\frac{1}{2}}\left(  \int_{\mathbb{R}^{N}}\left\vert u\right\vert
^{2}\left\vert x\right\vert ^{2}dx\right)  ^{\frac{1}{2}}-\frac{N}{2}%
\int_{\mathbb{R}^{N}}\left\vert u\right\vert ^{2}dx\geq d_{1}^{2}(u,E_{HUP}).
\label{SHUP}%
\end{equation}
Moreover, the inequality is sharp and the equality can be attained by
nontrivial functions $u\notin E_{HUP}$.
\end{theorema}

As a consequence, we can deduce from the above Theorem \ref{A} that%
\[
\theta_{2}(u)\geq N\left(  \int_{\mathbb{R}^{N}}\left\vert u\right\vert
^{2}dx\right)  d_{1}^{2}(u,E_{HUP})+d_{1}^{4}(u,E_{HUP})
\]
and the inequality is sharp with optimal constants $C_{1}=N$ and $C_{2}=1$,
and can be attained by nontrivial functions $u\notin E_{HUP}$.

Since the inequality (\ref{SHUP}) in Theorem \ref{A} can be attained by
nontrivial optimizers, we can once again ask for its stability. More
explicitly, is it possible to set up some stability version for (\ref{SHUP})?
This question was investigated very recently in \cite{LLR25}. The approach in
\cite{LLR25} is very similar to the method in \cite{CFLL24}. Indeed, the
authors in \cite{LLR25} combined the HUP identity (\ref{HUPI}) with certain
improved version of the Poincar\'{e} inequality with Gaussian type measure.
More precisely, in order to get the stability of (\ref{SHUP}), the authors in
\cite{LLR25} proved the following improved version of the sharp Gaussian
Poincar\'{e} inequality:

\begin{theorema}
\label{B}For all $\lambda\neq0$%
\[
\int_{\mathbb{R}^{N}}|\nabla u|^{2}e^{-\frac{1}{2\left\vert \lambda\right\vert
^{2}}\left\vert x\right\vert ^{2}}dx\geq\frac{1}{\left\vert \lambda\right\vert
^{2}}\inf_{c,\mathbf{d}}\int_{\mathbb{R}^{N}}\left(  \left\vert u-c\right\vert
^{2}+\left\vert u-c-\mathbf{d}\cdot x\right\vert ^{2}\right)  e^{-\frac
{1}{2\left\vert \lambda\right\vert ^{2}}\left\vert x\right\vert ^{2}%
}dx\text{.}%
\]
Also, when $\lambda=1$, then
\begin{align*}
&  \int_{\mathbb{R}^{N}}\left\vert \nabla u\right\vert ^{2}\frac{e^{-\frac
{1}{2}\left\vert x\right\vert ^{2}}}{\left(  2\pi\right)  ^{\frac{N}{2}}%
}dx-\int_{\mathbb{R}^{N}}\left\vert u-\int_{\mathbb{R}^{N}}u\frac{e^{-\frac
{1}{2}\left\vert x\right\vert ^{2}}}{\left(  2\pi\right)  ^{\frac{N}{2}}%
}dx\right\vert ^{2}\frac{e^{-\frac{1}{2}\left\vert x\right\vert ^{2}}}{\left(
2\pi\right)  ^{\frac{N}{2}}}dx\\
&  \geq\frac{1}{2}\int_{\mathbb{R}^{N}}\left\vert \nabla u-\int_{\mathbb{R}%
^{N}}ux\frac{e^{-\frac{1}{2}\left\vert x\right\vert ^{2}}}{\left(
2\pi\right)  ^{\frac{N}{2}}}dx\right\vert ^{2}\frac{e^{-\frac{1}{2}\left\vert
x\right\vert ^{2}}}{\left(  2\pi\right)  ^{\frac{N}{2}}}dx\\
&  \geq\int_{\mathbb{R}^{N}}\left\vert u-\int_{\mathbb{R}^{N}}u\frac
{e^{-\frac{1}{2}\left\vert x\right\vert ^{2}}}{\left(  2\pi\right)  ^{\frac
{N}{2}}}dx-\left(  \int_{\mathbb{R}^{N}}ux\frac{e^{-\frac{1}{2}\left\vert
x\right\vert ^{2}}}{\left(  2\pi\right)  ^{\frac{N}{2}}}dx\right)  \cdot
x\right\vert ^{2}\frac{e^{-\frac{1}{2}\left\vert x\right\vert ^{2}}}{\left(
2\pi\right)  ^{\frac{N}{2}}}dx.
\end{align*}

\end{theorema}

Then, with the same approach as in \cite{CFLL24}, the authors proved in
\cite{LLR25} the following stability of (\ref{SHUP}):

\begin{theorema}
\label{C}For all $u\in S_{0}:$%
\[
\theta_{1}\left(  u\right)  \geq d_{2}^{2}(u,F)
\]
and%
\[
\theta_{1}\left(  u\right)  -d_{1}^{2}(u,E_{HUP})\geq d_{1}^{2}(u,F).
\]

\end{theorema}

Here
\[
F:=\left\{  \left(  c+\mathbf{d}\cdot x\right)  e^{-\beta\left\vert
x\right\vert ^{2}}:c\in\mathbb{R}\text{, }\mathbf{d}\in\mathbb{R}^{N}\text{,
}\beta>0\right\}
\]
is the set of all the optimizers of (\ref{SHUP}), and
\[
d_{2}(u,F):=\inf_{c,\mathbf{d},\beta>0}\left(  \int_{\mathbb{R}^{N}}\left\vert
u-ce^{-\beta\left\vert x\right\vert ^{2}}\right\vert ^{2}+\left\vert u-\left(
c+\mathbf{d}\cdot x\right)  e^{-\beta\left\vert x\right\vert ^{2}}\right\vert
^{2}dx\right)  ^{\frac{1}{2}}.
\]

\medskip

The problem of finding the best constant of HUP when one replaces $u$ in
\eqref{HUP} by a divergence-free vector field $\mathbf{U}$ was posed by Maz'ya
in \cite[Section 3.9]{M}. Indeed, motivated by questions in hydrodynamics,
Maz'ya raised the following open question in \cite[Section 3.9]{M}: determine
the best constant $\mu^{\ast}(N)$ in the following inequality
\begin{equation}
\int_{\mathbb{R}^{N}}|\nabla\mathbf{U}|^{2}dx\int_{\mathbb{R}^{N}}\left\vert
x\right\vert ^{2}\left\vert \mathbf{U}\right\vert ^{2}dx\geq\mu^{\ast
}(N)\left(  \int_{\mathbb{R}^{N}}\left\vert \mathbf{U}\right\vert
^{2}dx\right)  ^{2},\quad\forall\mathbf{U}\in\left(  C_{0}^{\infty}%
(\mathbb{R}^{N})\right)  ^{N},\quad\mathrm{\operatorname{div}}\mathbf{U}=0.
\label{HUP_div}%
\end{equation}
In \cite{CFL22}, the authors established several results, including the proof
that $\mu^{\star}(2)=4$ and therefore resolving Maz'ya question for the case
$N=2$. The method in \cite{CFL22} relies on the fact that in $\mathbb{R}^{2}$,
the divergence-free vector fields are isometrically isomorphic to the
curl-free vector fields. More precisely, any divergence-free vector
$\mathbf{U}$ can be written in the form $\mathbf{U}=(-u_{x_{2}},u_{x_{1}})$
where $u$ is a scalar field. Also, if $\mathbf{U}\in\left(  C_{0}^{\infty
}(\mathbb{R}^{2})\right)  ^{2}$ then $u\in C_{0}^{\infty}(\mathbb{R}^{2})$.
Therefore, applying integration by parts yields%

\[
\int_{\mathbb{R}^{2}}|\nabla\mathbf{U}|^{2}dx=\int_{\mathbb{R}^{2}}|\Delta
u|^{2}dx.
\]
Hence \eqref{HUP_div} is equivalent to the following second order HUP-type
inequality in $\mathbb{R}^{2}$:%

\begin{equation}
\int_{\mathbb{R}^{2}}|\Delta u|^{2}dx\int_{\mathbb{R}^{2}}\left\vert
x\right\vert ^{2}|\nabla u|^{2}dx\geq\mu^{\star}(2)\left(  \int_{\mathbb{R}%
^{2}}|\nabla u|^{2}dx\right)  ^{2}. \label{HUP_div_scalar}%
\end{equation}
Then, by using spherical harmonic decomposition, the authors showed in
\cite{CFL22} that $\mu^{\star}(2)=4$ is sharp in (\ref{HUP_div_scalar}) and is
achieved by the Gaussian profiles of the form $\displaystyle u(x)=\alpha
e^{-\beta\left\vert x\right\vert ^{2}}$, $\beta>0$. Equivalently, $\mu^{\star
}(2)=4$ is sharp in (\ref{HUP_div}) and is attained by the vector fields of
the form $\displaystyle\mathbf{U}\left(  x\right)  =\left(  -\alpha
e^{-\beta\left\vert x\right\vert ^{2}}x_{2},\alpha e^{-\beta\left\vert
x\right\vert ^{2}}x_{1}\right)  $, $\beta>0$, $\alpha\in\mathbb{R}$. Very
recently, Hamamoto answered Maz'ya's open question in the remaining case
$N\geq3$ in \cite{Ham21}. More precisely, Hamamoto applied the
poloidal-toroidal decomposition to prove that $\mu^{\ast}(N)=\frac{1}%
{4}\left(  \sqrt{N^{2}-4\left(  N-3\right)  }+2\right)  ^{2}$ when $N\geq3$.

It is worthy to mention that we also proved in \cite{CFL22} the following
sharp HUP for curl-free vector fields: for $\mathbf{U}\in\left(  C_{0}%
^{\infty}(\mathbb{R}^{N})\right)  ^{N}$, $\mathrm{\operatorname{curl}%
}\mathbf{U}=0$, there holds
\begin{equation}
\int_{\mathbb{R}^{N}}|\nabla\mathbf{U}|^{2}dx\int_{\mathbb{R}^{N}}\left\vert
x\right\vert ^{2}\left\vert \mathbf{U}\right\vert ^{2}dx\geq\left(  \frac
{N+2}{2}\right)  ^{2}\left(  \int_{\mathbb{R}^{N}}\left\vert \mathbf{U}%
\right\vert ^{2}dx\right)  ^{2}, \label{HUP_curl}%
\end{equation}
by improving the best constant $\frac{N^{2}}{4}$ which corresponds to scalar
fields in \eqref{HUP}. Indeed, since in this case we can write $\mathbf{U}%
=\nabla u$ for some scalar potential $u:\mathbb{R}^{N}\mapsto\mathbb{C}$,
(\ref{HUP_curl}) is equivalent to the following second order CKN inequality:
\begin{equation}
\int_{\mathbb{R}^{N}}|\Delta u|^{2}dx\int_{\mathbb{R}^{N}}\left\vert
x\right\vert ^{2}|\nabla u|^{2}dx\geq\frac{(N+2)^{2}}{4}\left(  \int
_{\mathbb{R}^{N}}|\nabla u|^{2}dx\right)  ^{2}. \label{HUP_curl_scalar}%
\end{equation}
Then, by using spherical harmonics decomposition, we proved in \cite{CFL22}
that the constant $\displaystyle\frac{(N+2)^{2}}{4}$ is optimal in
\eqref{HUP_curl_scalar} and is attained for Gaussian profiles of the form
$\displaystyle u(x)=\alpha e^{-\beta\left\vert x\right\vert ^{2}}$, $\alpha
\in\mathbb{R},\beta>0$. Therefore, the constant $\displaystyle\left(
\frac{N+2}{2}\right)  ^{2}$ is sharp in (\ref{HUP_curl}) and is achieved by
the vector fields $\displaystyle\mathbf{U}\left(  x\right)  =\alpha
e^{-\beta\left\vert x\right\vert ^{2}}x$, $\alpha\in\mathbb{R},\beta>0$.
Simpler proofs of (\ref{HUP_curl_scalar}) was recently provided in
\cite{CFL23, DN23}. Moreover, in \cite{CFL23}, the authors provided a simple
proof for an improvement of (\ref{HUP_curl_scalar}), and used it to show that
the Gaussian profiles $\alpha e^{-\beta\left\vert x\right\vert ^{2}}$,
$\alpha\in\mathbb{R}$, $\beta>0$, are all the optimizers of
(\ref{HUP_curl_scalar}).

Now, let $E_{SHUP}:=\left\{  \alpha e^{-\beta\left\vert x\right\vert ^{2}%
}:\alpha\in\mathbb{R}\text{, }\beta>0\right\}  $ be the set of all optimizers
for (\ref{HUP_curl_scalar}). It is again natural to ask about the stabilities
of the second order HUP (\ref{HUP_curl_scalar}) and the HUP with curl-free
fields (\ref{HUP_curl}). This problem was studied very recently in
\cite{DN25}. More clearly, the authors in \cite{DN25} proved the following
stability of the second order HUP (\ref{HUP_curl_scalar}):%
\begin{equation}
\delta_{1}(u)\geq\frac{1}{768}\inf\left\{  \left\Vert \nabla\left(  u-u^{\ast
}\right)  \right\Vert _{2}^{2}:u^{\ast}\in E_{SHUP}\text{ and }\left\Vert
\nabla u\right\Vert _{2}^{2}=\left\Vert \nabla u^{\ast}\right\Vert _{2}%
^{2}\right\}  . \label{HD}%
\end{equation}
Here
\[
\delta_{1}(u)=\left(  \int_{\mathbb{R}^{N}}|\Delta u|^{2}dx\right)
^{1/2}\left(  \int_{\mathbb{R}^{N}}\left\vert x\right\vert ^{2}|\nabla
u|^{2}dx\right)  ^{1/2}-\dfrac{N+2}{2}\int_{\mathbb{R}^{N}}|\nabla u|^{2}dx
\]
is the second order HUP deficit. The approach in \cite{DN25} is rather similar
to the one in \cite{CFLL24}. However, because it involves second-order terms,
the analysis becomes significantly more intricate. More precisely, first, the
authors in \cite{DN25} set up an improvement of (\ref{HUP_curl_scalar}) for
functions that are orthogonal to radial functions and use it to establish an
upper bound for $\int_{\mathbb{R}^{N}}|\nabla u_{o}|^{2}dx$ where $u_{o}$ is
the odd part of $u$. This implies that when the deficit $\delta_{1}(u)$ is
small, then the function is almost even. Now, to estimate the stability for
even functions, the authors use the following identity%
\[
\int_{\mathbb{R}^{N}}|\Delta u|^{2}dx+\int_{\mathbb{R}^{N}}\left\vert
x\right\vert ^{2}|\nabla u|^{2}dx-\left(  N+2\right)  \int_{\mathbb{R}^{N}%
}|\nabla u|^{2}dx=\int_{\mathbb{R}^{N}}\left\Vert \nabla^{2}v-x\otimes\nabla
v\right\Vert _{HS}^{2}e^{-\left\vert x\right\vert ^{2}}dx
\]
where $v=ue^{\frac{\left\vert x\right\vert ^{2}}{2}\text{ }}$. Here
$\nabla^{2}v$ is the Hessian matrix of $v$, $x\otimes\nabla v$ denotes the
matrix $\left(  x_{i}\partial_{j}v\right)  _{i,j}$ and $\left\Vert
A\right\Vert _{HS}$ is the Hilbert-Schmidt norm of the matrix $A$. However,
there is no Poincar\'{e} inequality with Gaussian type measure for the
(second-order) term $\left\Vert \nabla^{2}v-x\otimes\nabla v\right\Vert
_{HS}^{2}$. In order to get a Poincar\'{e} type inequality for this term, the
main approach in \cite{DN25} is spectral analysis of the Ornstein-Uhlenbeck
type operator associated with the Gaussian weight and Hermite polynomials,
which is standard because of the Gaussian measure $e^{-\left\vert x\right\vert
^{2}}dx$. However, this process is very long and very complicated since the
term $\left\Vert \nabla^{2}v-x\otimes\nabla v\right\Vert _{HS}^{2}$ is rather
unusal. We also note that the stability constant of (\ref{HD}) in \cite{DN25}
is $\frac{1}{768}$ which is far from being optimal. Indeed, in our Theorem
\ref{T2.1}, we show that $\liminf$ of the sharp stability constants of
(\ref{HD}) is at least $\frac{1}{2}$ as $N\rightarrow\infty$.

The primary goal of this article is to present a completely different and much
simpler approach to the stability of the second order HUP and the HUP with
curl-free fields. Our method relies on tools in harmonic analysis such as
Fourier transform of radial functions and spherical harmonics decomposition.
It is worth mentioning that spherical harmonics represent one of the most
elegant and powerful mathematical tools in analysis. Their applications extend
across both pure and applied mathematics, with far-reaching implications for
numerous disciplines. In particular, spherical harmonics decomposition has
been used both implicitly and explicitly in establishing sharp constants in
various functional inequalities as well as in the study of isoperimetric
inequalities and related geometric problems. See \cite{Beckner93, CL90, FL12,
FLS08, Lieb83, LL01, LSSY02}, to name just a few. Moreover, spherical
harmonics have been frequently applied to investigate stability, rigidity, and
quantitative versions of geometric and functional inequalities. See
\cite{BDV15, CFMP09, FL121, Fug89, Mag12}, for instance. We also refer the
interested reader to the excellent books \cite{Stein70, Stein93} in which
Stein provided new perspectives on spherical harmonics, established deep
connections between spherical harmonic decompositions and complex analysis,
and developed methods that extended classical techniques for spherical
harmonics to more general settings.

Let%
\[
C\left(  N,k\right)  =\inf_{u\text{ is radial}}\frac{\int_{\mathbb{R}^{N+2k}%
}|\Delta u|^{2}dx+\int_{\mathbb{R}^{N+2k}}\left[  |\nabla u|^{2}\left\vert
x\right\vert ^{2}-2k|u|^{2}\right]  dx}{\int_{\mathbb{R}^{N+2k}}|\nabla
u|^{2}dx}-\left(  N+2\right)
\]
and
\[
C\left(  N\right)  =\inf_{k\in\mathbb{N}}C\left(  N,k\right)  .
\]
Let
\[
K\left(  N\right)  =\inf_{k\in\mathbb{N}}\left(  \sqrt{\left(  N+2k\right)
^{2}-8k}-N\right)  =\dfrac{4N-4}{\sqrt{N^{2}+4N-4}+N}.
\]
Note that $K\left(  N\right)  \rightarrow2$ as $N\rightarrow\infty$.

Now, let $X$ be the completion of $C_{0}^{\infty}(\mathbb{R}^{N})$ under the
seminorm $\left(  \int_{\mathbb{R}^{N}}|\Delta u|^{2}dx+\int_{\mathbb{R}^{N}%
}\left\vert x\right\vert ^{2}|\nabla u|^{2}dx\right)  ^{\frac{1}{2}}$. Our
first main result of this article is the following stability version of
(\ref{HUP_curl_scalar}):

\begin{theorem}
\label{T1}For all $u\in X$, we have
\begin{equation}
\delta_{1}(u)\geq\dfrac{C\left(  N\right)  }{2}\inf_{u^{\ast}\in E_{SHUP}%
}\left\Vert \nabla\left(  u-u^{\ast}\right)  \right\Vert _{2}^{2}.
\label{T1.1}%
\end{equation}
Moreover, the constant $\dfrac{C\left(  N\right)  }{2}$ is sharp.
\end{theorem}

We also have the following estimates on the stability constants:

\begin{theorem}
\label{L1}We have $C\left(  N,0\right)  =0$ and
\[
\sqrt{\left(  N+2k\right)  ^{2}-8k}-N\leq C\left(  N,k\right)  \leq2k\text{
}\forall k\geq1.
\]
As a consequence, we obtain
\[
K(N) \leq C\left(  N\right) = C\left(  N,1\right) \leq \frac{2N}{N+2}
\]
and
\[
\lim_{N\rightarrow\infty}C\left(  N\right)  =2\text{.}%
\]

\end{theorem}

We note that in a recent paper \cite{HY25}, by refining the analysis in this paper, the authors showed that $C(N)=C(N,1)=\sqrt{N^2+4N-4}-N$. Moreover, this sharp stability constant can be attained.

Let us briefly describe our approach in proving Theorem \ref{T1}. First, we
will linearize the deficit $\delta_{1}$ and study the stability problem with
the following deficit:
\[
\delta_{2}(u):=\int_{\mathbb{R}^{N}}|\Delta u|^{2}dx+\int_{\mathbb{R}^{N}%
}\left\vert x\right\vert ^{2}|\nabla u|^{2}dx-\left(  N+2\right)
\int_{\mathbb{R}^{N}}|\nabla u|^{2}dx.
\]
Obviously, by the AM-GM inequality%
\[
\delta_{2}(u)\geq2\delta_{1}(u)\geq0.
\]
We note that though $\delta_{2}(u)$ is not scaling-invariant, it is much more
tractable when we apply the spherical harmonic decompositions on $\delta
_{2}(u)$ than one does that on $\delta_{1}(u)$. With this new deficit function
$\delta_{2}(u)$, we will prove that

\begin{theorem}
\label{T2}For all $u\in X$, we have
\[
\delta_{2}(u)\geq C\left(  N\right)  \inf_{c}\int_{\mathbb{R}^{N}}\left\vert
\nabla\left(  u-ce^{-\left\vert x\right\vert ^{2}/2}\right)  \right\vert
^{2}dx.
\]
Moreover, the constant $C\left(  N\right)  $ is sharp.
\end{theorem}

Now, to prove the stability of the second order HUP with the deficit
$\delta_{2}$, we will use a completely different approach than the ones in
\cite{CFLL24, DN25}. Indeed, instead of using the term $\int_{\mathbb{R}^{N}%
}\left\Vert \nabla^{2}v-x\otimes\nabla v\right\Vert _{HS}^{2}e^{-\left\vert
x\right\vert ^{2}}dx$ and the Hermite polynomials decomposition of the
Ornstein-Uhlenbeck type operator as in \cite{DN25} or trying to apply the
Poincar\'{e} inequality directly as in \cite{CFLL24}, we will first apply the
standard spherical harmonics decomposition to $\delta_{2}(u)$. The key idea
here is that after performing the spherical harmonics decomposition to write
$\delta_{2}(u)$ as a series of radial terms, we will apply the Fourier
transform, or more exactly, the following identity proved recently in
\cite{DDLL24}:
\[
\dfrac{1}{|\mathbb{S}^{N+1}|}\int_{\mathbb{R}^{N+2}}|\Delta^{\alpha/2}%
v|^{2}dx=\dfrac{1}{|\mathbb{S}^{N-1}|}\int_{\mathbb{R}^{N}}|\Delta
^{\alpha/2+1/2}u|^{2}dx,
\]
where $v(x)=V(\left\vert x\right\vert ):=\frac{U^{\prime}(\left\vert
x\right\vert )}{\left\vert x\right\vert }$, and $u(x):=U(\left\vert
x\right\vert )$, to those radial terms. With the above identity, we can
convert the second order HUP on $\mathbb{R}^{N+2k}$ to the first order HUP on
$\mathbb{R}^{N+2k+2}$. Therefore, once again we can follow the strategy in
\cite{CFLL24} to establish the stability results for $\delta_{2}(u)$. Then, by
performing the scaling arguments, we obtain the stability of
(\ref{HUP_curl_scalar}) with the deficit $\delta_{1}(u)$.

As a consequence of our main results, we obtain the following stability
version of the second order HUP which significantly enhances the results in
\cite{DN25}:

\begin{theorem}
\label{T2.1}For all $u\in X$, we have
\begin{equation}
\delta_{1}(u)\geq\dfrac{C\left(  N\right)  }{4}\inf\left\{  \left\Vert
\nabla\left(  u-u^{\ast}\right)  \right\Vert _{2}^{2}:u^{\ast}\in
E_{SHUP}\text{ and }\left\Vert \nabla u\right\Vert _{2}^{2}=\left\Vert \nabla
u^{\ast}\right\Vert _{2}^{2}\right\}  . \label{T1.2}%
\end{equation}

\end{theorem}

It is obvious to see that by Theorem \ref{L1}, the $\lim\inf$ of the sharp
stability constants of (\ref{T1.2}) is at least $\frac{1}{2}$ as
$N\rightarrow\infty$. Therefore, (\ref{T1.2}) substantially improves (\ref{HD}).

As another direct consequence of our main results, we also obtain the
stability of the HUP with curl-free vector field. More precisely, let%
\[
\delta_{3}(\mathbf{U}):=\left(  \int_{\mathbb{R}^{N}}|\nabla\mathbf{U}%
|^{2}dx\right)  ^{\frac{1}{2}}\left(  \int_{\mathbb{R}^{N}}\left\vert
x\right\vert ^{2}\left\vert \mathbf{U}\right\vert ^{2}dx\right)  ^{\frac{1}%
{2}}-\left(  \frac{N+2}{2}\right)  \left(  \int_{\mathbb{R}^{N}}\left\vert
\mathbf{U}\right\vert ^{2}dx\right)
\]
and%
\[
E_{CFHUP}:=\left\{  \alpha e^{-\beta\left\vert x\right\vert ^{2}}x,\text{
}\alpha\in\mathbb{R},\beta>0\right\}  .
\]
Then

\begin{theorem}
\label{T3}For all smooth curl-free field $\mathbf{U}$, i.e.
$\operatorname{curl}\mathbf{U}=\mathbf{0}$, we have
\[
\delta_{3}(\mathbf{U})\geq\dfrac{C\left(  N\right)  }{2}\inf_{\mathbf{U}%
^{\ast}\in E_{CFHUP}}\left\Vert \mathbf{U-U}^{\ast}\right\Vert _{2}^{2}.
\]
The constant $\dfrac{C\left(  N\right)  }{2}$ is sharp.

Also, for all smooth curl-free field $\mathbf{U}$, we have
\[
\delta_{3}(\mathbf{U})\geq\dfrac{C\left(  N\right)  }{4}\inf\left\{
\left\Vert \mathbf{U-U}^{\ast}\right\Vert _{2}^{2}:\mathbf{U}^{\ast}\in
E_{CFHUP}\text{ and }\left\Vert \mathbf{U}\right\Vert _{2}^{2}=\left\Vert
\mathbf{U}^{\ast}\right\Vert _{2}^{2}\right\}  .
\]

\end{theorem}

Another consequence of our main results is the following second order
Poincar\'{e} type inequality with Gaussian measure that is of independent interest:

\begin{theorem}
\label{T4}For all $u\in X$, we have%
\[
\int_{\mathbb{R}^{N}}\left\Vert \nabla^{2}u-x\otimes\nabla u\right\Vert
_{HS}^{2}e^{-\left\vert x\right\vert ^{2}}dx\geq C\left(  N\right)  \inf
_{c}\int_{\mathbb{R}^{N}}\left\vert \nabla u-\left(  u-c\right)  x\right\vert
^{2}e^{-\left\vert x\right\vert ^{2}}dx.
\]
Moreover, the constant $C\left(  N\right)  $ is sharp.
\end{theorem}

We note that in \cite{DN25}, Theorem \ref{T4} has been investigated for even
functions. However, there is no estimate for arbitrary functions in
\cite{DN25}.

Our paper is organized as follows: In Section 2, we will discuss the spherical
harmonics decomposition and some useful lemmas. Also, to briefly illustrate
for our main idea, we present how to use spherical harmonics decomposition to
establish a symmetrization principle for the stability of the first order HUP
and give a new proof of Theorem \ref{A}. Theorem \ref{T3} and Theorem \ref{T4}
will also be proved in this section. Our main results\ (Theorem \ref{T1},
Theorem \ref{L1}, Theorem \ref{T2} and Theorem \ref{T2.1}) will be studied in
Section 3.

\section{Preliminaries and some useful results}

\subsection{Some useful lemmas}

Let $E_{SHUP}=\{\alpha e^{-\frac{\beta}{2}|x|^{2}},$ $\alpha\in\mathbb{R},$
$\beta>0\}$ be the set of all the optimizers of the second order HUP
(\ref{HUP_curl_scalar}). Then, we have

\begin{lemma}
\label{PropOfE} The set $E_{SHUP}$ is closed under the seminorm $\Vert
\nabla(\cdot)\Vert_{L^{2}(\mathbb{R}^{N})}$.

\end{lemma}

\begin{proof}
Let $u$ be an arbitrary function in $X$. Assume that there exists a sequence
$\{v_{j}\}_{j}\subset E_{SHUP}$ such that
\[
\Vert\nabla(v_{j}-u)\Vert_{L^{2}(\mathbb{R}^{N})}\rightarrow0
\]
as $j\rightarrow\infty$. We will show that $u\in$ $E_{SHUP}$. If $\Vert\nabla
u\Vert_{L^{2}(\mathbb{R}^{N})}=0$, then $u=0$ and $u\in$ $E_{SHUP}$.
Therefore, we can assume that $\Vert\nabla u\Vert_{L^{2}(\mathbb{R}^{N})}%
\neq0$. Moreover, WLOG, we assume that $\Vert\nabla u\Vert_{L^{2}%
(\mathbb{R}^{N})}=1$. Since
\[
\Vert\nabla(u-v_{j})\Vert_{L^{2}(\mathbb{R}^{N})}\geq\left\vert \Vert\nabla
u\Vert_{L^{2}(\mathbb{R}^{N})}-\Vert\nabla(v_{j})\Vert_{L^{2}(\mathbb{R}^{N}%
)}\right\vert =\left\vert 1-\Vert\nabla(v_{j})\Vert_{L^{2}(\mathbb{R}^{N}%
)}\right\vert ,
\]
we have that
\[
\dfrac{1}{2}\leq\Vert\nabla(v_{j})\Vert_{L^{2}(\mathbb{R}^{N})}\leq\dfrac
{3}{2}%
\]
for $j$ large enough. Assume that $v_{j}=\alpha_{j}e^{-\frac{\beta_{j}}%
{2}|x|^{2}}$. Then by computing directly, we get
\[
\int_{\mathbb{R}^{N}}|\nabla v_{j}|^{2}dx=\alpha_{j}^{2}\beta_{j}^{2}%
\int_{\mathbb{R}^{N}}e^{-\beta_{j}|x|^{2}}|x|^{2}dx=\alpha_{j}^{2}\beta
_{j}^{1-\frac{N}{2}}\int_{\mathbb{R}^{N}}e^{-|x|^{2}}|x|^{2}dx=\dfrac{N}{2}%
\pi^{N/2}\alpha_{j}^{2}\beta_{j}^{1-\frac{N}{2}}.
\]
This implies that there exist constants $M,m>0$ such that
\begin{equation}
m^{2}\leq\alpha_{j}^{2}\beta_{j}^{1-\frac{N}{2}}\leq M^{2}, \label{relationab}%
\end{equation}
for $j$ large enough. We claim that both $\{\alpha_{j}\}_{j}$ and
$\{b_{j}\}_{j}$ are bounded and therefore we can assume that $\alpha
_{j}\rightarrow\alpha$, and $\beta_{j}\rightarrow\beta>0$ up to a subsequence.
To verify our claim, we will use an argument from \cite{DN23}. First, we note
that
\[
\int_{\mathbb{R}^{N}}|\nabla u-\nabla v_{j}|^{2}dx=1+\alpha_{j}^{2}\beta
_{j}^{1-N/2}\int_{\mathbb{R}^{N}}|x|^{2}e^{-|x|^{2}}dx+2\alpha_{j}\beta
_{j}^{1/2}\int_{\mathbb{R}^{N}}\nabla u\cdot(\beta_{j}^{1/2}x)e^{-|\beta
_{j}^{1/2}x|^{2}/2}dx.
\]
Up to a subsequence, by Bolzano-Weierstrass theorem, we are able to reduce to
three cases:

\begin{itemize}
\item[i)] $\displaystyle\lim_{j \rightarrow\infty} \beta_{j}=\infty$;

\item[ii)] $\displaystyle\lim_{j \rightarrow\infty} \beta_{j}=0$;

\item[iii)] $\displaystyle\lim_{j \rightarrow\infty} \beta_{j} \in(0,\infty)$.
\end{itemize}

We will prove that the first two cases can not happen. Indeed, suppose that
$\displaystyle\lim_{j\rightarrow\infty}\beta_{j}=\infty$. For any $\epsilon
>0$, there exists a constant $R>0$ small enough such that $\int_{B_{R}}|\nabla
u|^{2}dx<\epsilon^{2}$, where $B_{R}:=\{x\in\mathbb{R}^{N}:|x|<R\}$. Then,
\begin{align*}
&  \left\vert \int_{\mathbb{R}^{N}}\nabla u\cdot(\beta_{j}^{1/2}%
x)e^{-|\beta_{j}^{1/2}x|^{2}/2}dx\right\vert \\
&  \leq\left(  \int_{B_{R}}+\int_{B_{R}^{c}}\right)  \left\vert \nabla
u\right\vert \left\vert \beta_{j}^{1/2}x\right\vert e^{-|\beta_{j}^{1/2}%
x|^{2}/2}dx\\
&  \leq\epsilon\left(  \int_{B_{R}}\left\vert \beta_{j}^{1/2}x\right\vert
^{2}e^{-|\beta_{j}^{1/2}x|^{2}}dx\right)  ^{1/2}+\left(  \int_{B_{R}^{c}%
}|\nabla u|^{2}dx\right)  ^{1/2}\left(  \int_{B_{R}^{c}}\left\vert \beta
_{j}^{1/2}x\right\vert ^{2}e^{-|\beta_{j}^{1/2}x|^{2}}dx\right)  ^{1/2}\\
&  \leq\epsilon\left(  \int_{B_{R}}\left\vert \beta_{j}^{1/2}x\right\vert
^{2}e^{-|\beta_{j}^{1/2}x|^{2}}dx\right)  ^{1/2}+\left(  \int_{B_{R}^{c}%
}\left\vert \beta_{j}^{1/2}x\right\vert ^{2}e^{-|\beta_{j}^{1/2}x|^{2}%
}dx\right)  ^{1/2}\\
&  \leq\beta_{j}^{-N/4}\epsilon\left(  \int_{B_{R\beta_{j}^{1/2}}}\left\vert
x\right\vert ^{2}e^{-|x|^{2}}dx\right)  ^{1/2}+\beta_{j}^{-N/4}\left(
\int_{B_{R\beta_{j}^{1/2}}^{c}}\left\vert x\right\vert ^{2}e^{-|x|^{2}%
}dx\right)  ^{1/2},
\end{align*}
which implies that
\begin{align*}
&  \left\vert \alpha_{j}\beta_{j}^{1/2}\int_{\mathbb{R}^{N}}\nabla
u\cdot(\beta_{j}^{1/2}x)e^{-|\beta_{j}^{1/2}x|^{2}/2}dx\right\vert \\
&  \leq M\epsilon\left(  \int_{B_{R\beta_{j}^{1/2}}}\left\vert x\right\vert
^{2}e^{-|x|^{2}}dx\right)  ^{1/2}+M\left(  \int_{B_{R\beta_{j}^{1/2}}^{c}%
}\left\vert x\right\vert ^{2}e^{-|x|^{2}}dx\right)  ^{1/2}.
\end{align*}
Letting $j\rightarrow\infty$, we obtain
\[
\limsup_{j\rightarrow\infty}\left\vert \alpha_{j}\beta_{j}^{1/2}%
\int_{\mathbb{R}^{N}}\nabla u\cdot(\beta_{j}^{1/2}x)e^{-|\beta_{j}^{1/2}%
x|^{2}/2}dx\right\vert \leq M\epsilon\left(  \int_{\mathbb{R}^{N}}%
|x|^{2}e^{-|x|^{2}}\right)  ^{1/2}.
\]
Since $\epsilon>0$ is arbitrary, we have
\[
\limsup_{j\rightarrow\infty}\left\vert \alpha_{j}\beta_{j}^{1/2}%
\int_{\mathbb{R}^{N}}\nabla u\cdot(\beta_{j}^{1/2}x)e^{-|\beta_{j}^{1/2}%
x|^{2}/2}dx\right\vert =0,
\]
or
\[
0=\displaystyle\lim_{j\rightarrow\infty}\int_{\mathbb{R}^{N}}|\nabla u-\nabla
v_{j}|^{2}dx>1,
\]
which is a contradiction.

Now assume that $\beta_{j}\rightarrow0$ as $j\rightarrow\infty$. For any
$\epsilon>0$, there exists a constant $R>0$ large enough such that
$\int_{B_{R}^{c}}|\nabla u|^{2}dx<\epsilon^{2}$. Thus,
\begin{align*}
&  \left\vert \int_{\mathbb{R}^{N}}\nabla u\cdot(\beta_{j}^{1/2}%
x)e^{-|\beta_{j}^{1/2}x|^{2}/2}dx\right\vert \\
&  \leq\left(  \int_{B_{R}^{c}}+\int_{B_{R}}\right)  \left\vert \nabla
u\right\vert \left\vert \beta_{j}^{1/2}x\right\vert e^{-|\beta_{j}^{1/2}%
x|^{2}/2}dx\\
&  \leq\epsilon\left(  \int_{B_{R}^{c}}\left\vert \beta_{j}^{1/2}x\right\vert
^{2}e^{-|\beta_{j}^{1/2}x|^{2}}dx\right)  ^{1/2}+\left(  \int_{B_{R}}|\nabla
u|^{2}dx\right)  ^{1/2}\left(  \int_{B_{R}}\left\vert \beta_{j}^{1/2}%
x\right\vert ^{2}e^{-|\beta_{j}^{1/2}x|^{2}}dx\right)  ^{1/2}\\
&  \leq\epsilon\left(  \int_{B_{R}^{c}}\left\vert \beta_{j}^{1/2}x\right\vert
^{2}e^{-|\beta_{j}^{1/2}x|^{2}}dx\right)  ^{1/2}+\left(  \int_{B_{R}%
}\left\vert \beta_{j}^{1/2}x\right\vert ^{2}e^{-|\beta_{j}^{1/2}x|^{2}%
}dx\right)  ^{1/2}\\
&  \leq\beta_{j}^{-N/4}\epsilon\left(  \int_{B_{R\beta_{j}^{1/2}}^{c}%
}\left\vert x\right\vert ^{2}e^{-|x|^{2}}dx\right)  ^{1/2}+\beta_{j}%
^{-N/4}\left(  \int_{B_{R\beta_{j}^{1/2}}}\left\vert x\right\vert
^{2}e^{-|x|^{2}}dx\right)  ^{1/2}.
\end{align*}
Therefore
\begin{align*}
&  \left\vert \alpha_{j}\beta_{j}^{1/2}\int_{\mathbb{R}^{N}}\nabla
u\cdot(\beta_{j}^{1/2}x)e^{-|\beta_{j}^{1/2}x|^{2}/2}dx\right\vert \\
&  \leq M\epsilon\left(  \int_{B_{R\beta_{j}^{1/2}}^{c}}\left\vert
x\right\vert ^{2}e^{-|x|^{2}}dx\right)  ^{1/2}+M\left(  \int_{B_{R\beta
_{j}^{1/2}}}\left\vert x\right\vert ^{2}e^{-|x|^{2}}dx\right)  ^{1/2}.
\end{align*}
Letting $j\rightarrow\infty$, we obtain
\[
\limsup_{j\rightarrow\infty}\left\vert \alpha_{j}\beta_{j}^{1/2}%
\int_{\mathbb{R}^{N}}\nabla u\cdot(\beta_{j}^{1/2}x)e^{-|\beta_{j}^{1/2}%
x|^{2}/2}dx\right\vert \leq M\epsilon\left(  \int_{\mathbb{R}^{N}}%
|x|^{2}e^{-|x|^{2}}\right)  ^{1/2}.
\]
Since $\epsilon>0$ is arbitrary, then
\[
\limsup_{j\rightarrow\infty}\left\vert \alpha_{j}\beta_{j}^{1/2}%
\int_{\mathbb{R}^{N}}\nabla u\cdot(\beta_{j}^{1/2}x)e^{-|\beta_{j}^{1/2}%
x|^{2}/2}dx\right\vert =0,
\]
or
\[
0=\displaystyle\lim_{j\rightarrow\infty}\int_{\mathbb{R}^{N}}|\nabla u-\nabla
v_{j}|^{2}dx>1,
\]
which is also a contradiction.

Therefore, there exists $\beta\in(0,\infty)$ such that $\displaystyle\lim
_{j\rightarrow\infty}\beta_{j}=\beta$. From \eqref{relationab}, $\{\alpha
_{j}\}$ is bounded. Then, up to a subsequence, we can assume that $\alpha
_{j}\rightarrow\alpha\in\mathbb{R}$, and $\beta_{j}\rightarrow\beta
\in(0,\infty)$.

By defining $v(x)=\alpha e^{-\frac{\beta}{2}|x|^{2}}$, from the dominated
convergence theorem, $\nabla v_{j}\rightarrow\nabla v$ in $L^{2}%
(\mathbb{R}^{N})$. Hence,
\[
\Vert\nabla(u-v)\Vert_{L^{2}(\mathbb{R}^{N})}\leq\Vert\nabla(u-v_{j}%
)\Vert_{L^{2}(\mathbb{R}^{N})}+\Vert\nabla(v-v_{j})\Vert_{L^{2}(\mathbb{R}%
^{N})}\rightarrow0,
\]
as $j\rightarrow\infty$. Then, $u-v$ is constant a.e. in $\mathbb{R}^{N}.$
Since $u\in X$, $u-v=0$, or $u=v=\alpha e^{-\frac{\beta}{2}|x|^{2}}\in
E_{SHUP}$. Equivalently, $E_{SHUP}$ is closed under the seminorm $\Vert
\nabla(\cdot)\Vert_{L^{2}(\mathbb{R}^{N})}$.


\end{proof}

We also recall the following result in \cite{DDLL24}.

Let $\mathcal{F}_{k}\left(  \phi\right)  \left(  \xi\right)  =\int
_{\mathbb{R}^{k}}\phi\left(  x\right)  e^{-2\pi ix\cdot\xi}dx$ be the Fourier
transform of $\phi$ on $\mathbb{R}^{k}$. Then, we have

\begin{lemma}
\label{T3.1}Let $u\left(  \cdot\right)  =U\left(  \left\vert \cdot\right\vert
\right)  $ and $v\left(  \cdot\right)  =V\left(  \left\vert \cdot\right\vert
\right)  $ be such that $\mathcal{F}_{N}\left(  u\right)  $ and $\mathcal{F}%
_{N+2}\left(  v\right)  $ are well-defined, $\lim_{r\rightarrow0^{+}}U\left(
r\right)  r^{N}=0$, $\lim_{r\rightarrow\infty}U\left(  r\right)  r^{\frac
{N-1}{2}}=0$, and $rV\left(  r\right)  =U^{\prime}\left(  r\right)  $ for
$r>0$. Then $\mathcal{F}_{N+2}\left(  v\right)  =-2\pi\mathcal{F}_{N}\left(
u\right)  $. As a consequence, for $\alpha\geq0$, if defined, then
\[
\int_{\mathbb{R}^{N}}|\Delta^{\alpha+\frac{1}{2}}u|^{2}dx=\frac{|\mathbb{S}%
^{N-1}|}{|\mathbb{S}^{N+1}|}\int_{\mathbb{R}^{N+2}}|\Delta^{\alpha}v|^{2}dx.
\]

\end{lemma}

\subsection{Spherical harmonics decomposition}

Let $N\geq2$. We implement the well-known idea of decomposing $u$ into
spherical harmonics as follows. We apply the coordinates transformation
$x\in\mathbb{R}^{N}\setminus\left\{  0\right\}  \mapsto(r,\sigma)\in
(0,\infty)\times\mathbb{S}^{N-1}$ (where $\mathbb{S}^{N-1}$ is the $\left(
N-1\right)  $-dimensional sphere with respect to the Hausdorff measure in
$\mathbb{R}^{N}$) and, for any function $u\in C_{0}^{\infty}(\mathbb{R}^{N})$,
we may expand $u$ in terms of spherical harmonics
\begin{equation}
u(x)=u(r\sigma)=\sum_{k=0}^{\infty}\sum_{\ell=1}^{\dim\mathcal{H}_{k}}%
u_{k\ell}(r)\phi_{k\ell}(\sigma), \label{harmonics}%
\end{equation}
where $\{\phi_{k\ell}\}$, $k\geq0$, $\ell=1,\ldots,\dim\mathcal{H}_{k}$, is an
orthonormal basis in $L^{2}(\mathbb{S}^{N-1})$ constituted by spherical
harmonic functions $\phi_{k\ell}$ of degree $k$, and $\mathcal{H}_{k}$ is the
subspace of spherical harmonics of degree $k$. Such $\phi_{k\ell}$ are smooth
eigenfunctions of the Laplace-Beltrami operator $-\Delta_{\mathbb{S}^{N-1}}$
with the corresponding eigenvalues $c_{k}=k(k+N-2)$, $k\geq0$, i.e.
\begin{align}
-\Delta_{\mathbb{S}^{N-1}}\phi_{k\ell}=c_{k}\phi_{k\ell},\quad\left\langle
\phi_{k\ell},\phi_{m\ell^{\prime}}\right\rangle _{L^{2}(\mathbb{S}^{N-1}%
)}=\delta_{km}\delta_{\ell\ell^{\prime}},  &  \quad\forall k,m\in
\mathbb{N},\label{eigen}\\
&  \ell=1,\ldots,\dim\mathcal{H}_{k},\ell^{\prime}=1,\ldots,\dim
\mathcal{H}_{m}\nonumber
\end{align}
where $\left\langle \cdot,\cdot\right\rangle _{L^{2}(\mathbb{S}^{N-1})}$
denotes the scalar product in $L^{2}(\mathbb{S}^{N-1})$ and $\delta_{kl}$ is
the Kronecker symbol. See for example \cite{Stein70, Stein93}. In fact,
without altering the following proofs and for sake of simplicity, we choose to
use a single index and agree to write
\[
u(x)=u(r\sigma)=\sum_{k=0}^{\infty}u_{k}(r)\phi_{k}(\sigma),
\]
where $-\Delta_{\mathbb{S}^{N-1}}\phi_{k}=c_{k}\phi_{k}$. Moreover, to
simplify some computations below, we will choose such that
\[
\int_{\mathbb{S}^{N-1}}\left\vert \phi_{k}(\sigma)\right\vert ^{2}%
d\sigma=\frac{|\mathbb{S}^{N-1+2k}|}{|\mathbb{S}^{N-1}|}%
\]
where $d\sigma$ is the normalized surface measure on $\mathbb{S}^{N-1}$.

The Fourier coefficients $\{u_{k}\}_{k}$ belong to $C_{0}^{\infty}%
([0,\infty))$ and satisfy $u_{k}(r)=O(r^{k})$, $u_{k}^{\prime}(r)=O(r^{k-1})$
as $r\rightarrow0$. Also, the following formulas hold:
\begin{equation}
\Delta u=\mathcal{R}_{2}u+\frac{1}{r^{2}}\Delta_{\mathbb{S}^{N-1}}%
u,\quad|\nabla u|^{2}=|\partial_{r}u|^{2}+\frac{|\nabla_{\mathbb{S}^{N-1}%
}u|^{2}}{r^{2}} \label{Laplacian}%
\end{equation}
where, $\mathcal{R}_{2}u:=\partial_{rr}^{2}u+\frac{N-1}{r}\partial_{r}u$ is
the radial Laplacian, $\partial_{r}$ and $\partial_{rr}^{2}$ are both partial
derivatives of first and second order with respect to the radial component
$r$, whereas $\Delta_{\mathbb{S}^{N-1}}$ and $\nabla_{\mathbb{S}^{N-1}}$
represent the Laplace-Beltrami operator and respectively the gradient operator
with respect to the metric tensor on $\mathbb{S}^{N-1}$. We also note that the
behavior of $u_{k}(r)$ as $r\rightarrow0$ allows us to consider the change of
variables $u_{k}(r):=r^{k}v_{k}(r)$, where $v_{k}\in C_{0}^{\infty}[0,\infty
)$. Therefore, we can represent $u$ in the form
\begin{equation}
u(x)=u(r\sigma)=\sum_{k=0}^{\infty}u_{k}(r)\phi_{k}(\sigma)=\sum_{k=0}%
^{\infty}v_{k}(r)r^{k}\phi_{k}(\sigma). \label{shd}%
\end{equation}
Then, by applying Bochner's relation for spherical harmonics and Lemma
\ref{T3.1}, we obtain the following identities (see \cite{Bec08, CFL22}):
\begin{align*}
\int_{\mathbb{R}^{N}}|\Delta u|^{2}dx  &  =\sum_{k=0}^{\infty}\int
_{\mathbb{R}^{N+2k}}|\Delta v_{k}(\left\vert x\right\vert )|^{2}dx\\
&  =\sum_{k=0}^{\infty}|\mathbb{S}^{N-1+2k}|\left(  \int_{0}^{\infty
}r^{N+2k-1}\left\vert v_{k}^{\prime\prime}\right\vert ^{2}dr+(N+2k-1)\int
_{0}^{\infty}r^{N+2k-3}\left\vert v_{k}^{\prime}\right\vert ^{2}dr\right) \\
&  =\sum_{k=0}^{\infty}\int_{\mathbb{R}^{N+2k+2}}|\nabla w_{k}(\left\vert
x\right\vert )|^{2}dx
\end{align*}
and%
\begin{align*}
\int_{\mathbb{R}^{N}}|\nabla u|^{2}dx  &  =\sum_{k=0}^{\infty}\int
_{\mathbb{R}^{N+2k}}|\nabla v_{k}(\left\vert x\right\vert )|^{2}dx\\
&  =\sum_{k=0}^{\infty}|\mathbb{S}^{N-1+2k}|\int_{0}^{\infty}r^{N+2k-1}%
\left\vert v_{k}^{\prime}\right\vert ^{2}dr\\
&  =\sum_{k=0}^{\infty}\int_{\mathbb{R}^{N+2k+2}}|w_{k}(\left\vert
x\right\vert )|^{2}dx,
\end{align*}
where $w_{k}(r)=\sqrt{\frac{|\mathbb{S}^{N+1+2k}|}{|\mathbb{S}^{N-1+2k}|}%
}\dfrac{v_{k}^{\prime}(r)}{r}$.

Moreover,
\begin{align*}
\int_{\mathbb{R}^{N}}\left\vert x\right\vert ^{2}|\nabla u|^{2}dx  &
=\sum_{k=0}^{\infty}\int_{\mathbb{R}^{N+2k}}|\nabla v_{k}(\left\vert
x\right\vert )|^{2}\left\vert x\right\vert ^{2}dx-2\sum_{k=0}^{\infty}%
k\int_{\mathbb{R}^{N+2k}}|v_{k}(\left\vert x\right\vert )|^{2}dx\\
&  =\sum_{k=0}^{\infty}|\mathbb{S}^{N-1+2k}|\int_{0}^{\infty}\left(
\left\vert v_{k}^{\prime}\right\vert ^{2}r^{2}-2k\left\vert v_{k}\right\vert
^{2}\right)  r^{N+2k-1}dr.
\end{align*}
Therefore
\begin{align*}
\int_{\mathbb{R}^{N}}\left\vert x\right\vert ^{2}|\nabla u|^{2}dx  &
=\sum_{k=0}^{\infty}|\mathbb{S}^{N-1+2k}|\left[  \dfrac{(N+2k)^{2}%
-8k}{(N+2k)^{2}}\int_{0}^{\infty}\left\vert v_{k}^{\prime}\right\vert
^{2}r^{N+2k+1}dr\right. \\
&  \left.  +\int_{0}^{\infty}\left(  \dfrac{2\sqrt{2k}}{N+2k}rv_{k}^{\prime
}(r)+\sqrt{2k}v_{k}(r)\right)  ^{2}r^{N+2k-1}dr\right] \\
&  \geq\sum_{k=0}^{\infty}\dfrac{(N+2k)^{2}-8k}{(N+2k)^{2}}|\mathbb{S}%
^{N-1+2k}|\int_{0}^{\infty}\left\vert v_{k}^{\prime}\right\vert ^{2}%
r^{N+2k+1}dr\\
&  \geq\sum_{k=0}^{\infty}\dfrac{(N+2k)^{2}-8k}{(N+2k)^{2}}\int_{\mathbb{R}%
^{N+2k}}|\nabla v_{k}(\left\vert x\right\vert )|^{2}\left\vert x\right\vert
^{2}dx\\
&  \geq\sum_{k=0}^{\infty}\dfrac{(N+2k)^{2}-8k}{(N+2k)^{2}}\int_{\mathbb{R}%
^{N+2k+2}}|w_{k}(\left\vert x\right\vert )|^{2}\left\vert x\right\vert ^{2}dx.
\end{align*}

\subsection{A symmetrization principle for the stability of the first order
HUP}

To illustrate for our main idea, we will prove a symmetrization principle for
the stability of the Hardy type inequalities using spherical harmonics
decomposition. As a consequence, we deduce a symmetrization principle for the
stability of HUP.

\begin{theorem}
Assume that $\forall N\geq N_{0}$, $\forall k\geq1$, $\exists C_{N}>0$ such
that
\[
A_{N+2k}\left(  r\right)  \geq A_{N}\left(  r\right)  +C_{N}W_{N}\left(
r\right)
\]
and
\[
\int_{\mathbb{R}^{N}}|\nabla u|^{2}dx\geq\int_{\mathbb{R}^{N}}A_{N}\left(
\left\vert x\right\vert \right)  \left\vert u\right\vert ^{2}dx\text{ }\forall
u\in C_{0,\text{radial}}^{\infty}(\mathbb{R}^{N}).
\]
Then $\forall N\geq N_{0}$, $\forall u\in C_{0}^{\infty}(\mathbb{R}^{N})$
\begin{equation}
\int_{\mathbb{R}^{N}}|\nabla u|^{2}dx-\int_{\mathbb{R}^{N}}A_{N}\left(
\left\vert x\right\vert \right)  \left\vert u\right\vert ^{2}dx\geq
C_{N}\left(  \int_{\mathbb{R}^{N}}\left\vert u\right\vert ^{2}W_{N}\left(
\left\vert x\right\vert \right)  dx-\frac{\left(  \int_{\mathbb{R}^{N}}%
uB_{N}\left(  \left\vert x\right\vert \right)  W_{N}\left(  \left\vert
x\right\vert \right)  dx\right)  ^{2}}{\int_{\mathbb{R}^{N}}B_{N}^{2}\left(
\left\vert x\right\vert \right)  W_{N}\left(  \left\vert x\right\vert \right)
dx}\right)  \text{ } \label{G}%
\end{equation}
if and only if $\forall N\geq N_{0}$, $\forall u\in C_{0,\text{radial}%
}^{\infty}(\mathbb{R}^{N})$%
\begin{equation}
\int_{\mathbb{R}^{N}}|\nabla u|^{2}dx-\int_{\mathbb{R}^{N}}A_{N}\left(
\left\vert x\right\vert \right)  \left\vert u\right\vert ^{2}dx\geq
C_{N}\left(  \int_{\mathbb{R}^{N}}\left\vert u\right\vert ^{2}W_{N}\left(
\left\vert x\right\vert \right)  dx-\frac{\left(  \int_{\mathbb{R}^{N}}%
uB_{N}\left(  \left\vert x\right\vert \right)  W_{N}\left(  \left\vert
x\right\vert \right)  dx\right)  ^{2}}{\int_{\mathbb{R}^{N}}B_{N}^{2}\left(
\left\vert x\right\vert \right)  W_{N}\left(  \left\vert x\right\vert \right)
dx}\right)  \text{.} \label{R}%
\end{equation}

\end{theorem}

\begin{proof}
Let $u\in C_{0}^{\infty}(\mathbb{R}^{N})$. Using the spherical harmonics
decomposition, we can represent $u$ in the form
\[
u(x)=u(r\sigma)=\sum_{k=0}^{\infty}u_{k}(r)\phi_{k}(\sigma)=\sum_{k=0}%
^{\infty}v_{k}(r)r^{k}\phi_{k}(\sigma).
\]
Then
\begin{align*}
\int_{\mathbb{R}^{N}}|\nabla u|^{2}dx  &  =\sum_{k=0}^{\infty}\int
_{\mathbb{R}^{N+2k}}|\nabla v_{k}(\left\vert x\right\vert )|^{2}dx\\
\int_{\mathbb{R}^{N}}A_{N}\left(  \left\vert x\right\vert \right)  \left\vert
u\right\vert ^{2}dx  &  =\sum_{k=0}^{\infty}\int_{\mathbb{R}^{N+2k}}%
A_{N}\left(  \left\vert x\right\vert \right)  |v_{k}|^{2}dx\\
\int_{\mathbb{R}^{N}}W_{N}\left(  \left\vert x\right\vert \right)  \left\vert
u\right\vert ^{2}dx  &  =\sum_{k=0}^{\infty}\int_{\mathbb{R}^{N+2k}}%
W_{N}\left(  \left\vert x\right\vert \right)  |v_{k}|^{2}dx
\end{align*}
and
\[
\int_{\mathbb{R}^{N}}uB_{N}\left(  \left\vert x\right\vert \right)
W_{N}\left(  \left\vert x\right\vert \right)  dx=\int_{\mathbb{R}^{N}}%
v_{0}B_{N}\left(  \left\vert x\right\vert \right)  W_{N}\left(  \left\vert
x\right\vert \right)  dx.
\]
Therefore, to prove (\ref{G}), it is enough to show that%
\begin{align*}
&  \int_{\mathbb{R}^{N}}|\nabla v_{0}|^{2}dx-\int_{\mathbb{R}^{N}}A_{N}\left(
\left\vert x\right\vert \right)  |v_{0}|^{2}dx\\
&  \geq C_{N}\left(  \int_{\mathbb{R}^{N}}|v_{0}|^{2}W_{N}\left(  \left\vert
x\right\vert \right)  dx-\frac{\left(  \int_{\mathbb{R}^{N}}v_{0}B_{N}\left(
\left\vert x\right\vert \right)  W_{N}\left(  \left\vert x\right\vert \right)
dx\right)  ^{2}}{\int_{\mathbb{R}^{N}}B_{N}^{2}\left(  \left\vert x\right\vert
\right)  W_{N}\left(  \left\vert x\right\vert \right)  dx}\right)
\end{align*}
and
\[
\int_{\mathbb{R}^{N+2k}}|\nabla v_{k}(\left\vert x\right\vert )|^{2}%
dx-\int_{\mathbb{R}^{N+2k}}A_{N}\left(  \left\vert x\right\vert \right)
|v_{k}|^{2}dx\geq C_{N}\int_{\mathbb{R}^{N+2k}}|v_{k}|^{2}W_{N}\left(
\left\vert x\right\vert \right)  dx.
\]
Note that $v_{0}$ is radial and that
\[
\int_{\mathbb{R}^{N+2k}}|\nabla v_{k}(\left\vert x\right\vert )|^{2}dx\geq
\int_{\mathbb{R}^{N+2k}}A_{N+2k}\left(  \left\vert x\right\vert \right)
|v_{k}|^{2}dx\geq\int_{\mathbb{R}^{N+2k}}\left(  A_{N}\left(  \left\vert
x\right\vert \right)  +C_{N}W_{N}\right)  |v_{k}|^{2}dx.
\]

\end{proof}

\begin{proof}
[Proof of Theorem \ref{A}]Denote
\[
\theta_{3}\left(  u\right)  :=\int_{\mathbb{R}^{N}}|\nabla u|^{2}%
dx+\int_{\mathbb{R}^{N}}\left\vert x\right\vert ^{2}\left\vert u\right\vert
^{2}dx-N\int_{\mathbb{R}^{N}}\left\vert u\right\vert ^{2}dx.
\]
Apply the above theorem for $A_{N}\left(  r\right)  =\left(  N-r^{2}\right)
$, $B_{N}\left(  r\right)  =e^{-\frac{1}{2}r^{2}}$, $W_{N}\left(  r\right)
=1$, we have that
\[
\theta_{3}\left(  u\right)  \geq C\left[  \int_{\mathbb{R}^{N}}\left\vert
u\right\vert ^{2}dx-\frac{\left(  \int_{\mathbb{R}^{N}}ue^{-\frac{1}%
{2}\left\vert x\right\vert ^{2}}dx\right)  ^{2}}{\int_{\mathbb{R}^{N}%
}e^{-\left\vert x\right\vert ^{2}}dx}\right]  \text{ }\forall u\in
C_{0}^{\infty}(\mathbb{R}^{N})
\]
if and only if
\[
\theta_{3}\left(  u\right)  \geq C\left[  \int_{\mathbb{R}^{N}}\left\vert
u\right\vert ^{2}dx-\frac{\left(  \int_{\mathbb{R}^{N}}ue^{-\frac{1}%
{2}\left\vert x\right\vert ^{2}}dx\right)  ^{2}}{\int_{\mathbb{R}^{N}%
}e^{-\left\vert x\right\vert ^{2}}dx}\right]  \text{ }\forall u\in
C_{0,\text{radial}}^{\infty}(\mathbb{R}^{N}).
\]

Let $u\in C_{0}^{\infty}(\mathbb{R}^{N})$. Using the spherical harmonics
decomposition, we can represent $u$ in the form
\[
u(x)=u(r\sigma)=\sum_{k=0}^{\infty}u_{k}(r)\phi_{k}(\sigma)=\sum_{k=0}%
^{\infty}v_{k}(r)r^{k}\phi_{k}(\sigma).
\]
Then
\begin{align*}
\int_{\mathbb{R}^{N}}|\nabla u|^{2}dx  &  =\sum_{k=0}^{\infty}\int
_{\mathbb{R}^{N+2k}}|\nabla v_{k}(\left\vert x\right\vert )|^{2}dx\\
\int_{\mathbb{R}^{N}}\left\vert x\right\vert ^{2}\left\vert u\right\vert
^{2}dx  &  =\sum_{k=0}^{\infty}\int_{\mathbb{R}^{N+2k}}\left\vert x\right\vert
^{2}|v_{k}|^{2}dx\\
\int_{\mathbb{R}^{N}}\left\vert u\right\vert ^{2}dx  &  =\sum_{k=0}^{\infty
}\int_{\mathbb{R}^{N+2k}}|v_{k}|^{2}dx.
\end{align*}
Also,%
\begin{align*}
\inf_{c}\int_{\mathbb{R}^{N}}\left\vert u-ce^{-\frac{1}{2}\left\vert
x\right\vert ^{2}}\right\vert ^{2}dx  &  =\int_{\mathbb{R}^{N}}\left\vert
u\right\vert ^{2}dx-\frac{\left(  \int_{\mathbb{R}^{N}}ue^{-\frac{1}%
{2}\left\vert x\right\vert ^{2}}dx\right)  ^{2}}{\int_{\mathbb{R}^{N}%
}e^{-\left\vert x\right\vert ^{2}}dx}\\
&  =\sum_{k=0}^{\infty}\int_{\mathbb{R}^{N+2k}}|v_{k}|^{2}dx-\frac{\left(
\int_{\mathbb{R}^{N}}v_{0}e^{-\frac{1}{2}\left\vert x\right\vert ^{2}%
}dx\right)  ^{2}}{\int_{\mathbb{R}^{N}}e^{-\left\vert x\right\vert ^{2}}dx}.
\end{align*}
Therefore, to find the largest positive constant $C$ such that
\[
\theta_{3}\left(  u\right)  \geq C\left[  \sum_{k=0}^{\infty}\int
_{\mathbb{R}^{N+2k}}|v_{k}|^{2}dx-\frac{\left(  \int_{\mathbb{R}^{N}}%
v_{0}e^{-\frac{1}{2}\left\vert x\right\vert ^{2}}dx\right)  ^{2}}%
{\int_{\mathbb{R}^{N}}e^{-\left\vert x\right\vert ^{2}}dx}\right]  ,
\]
it is enough to find $C$ such that%
\begin{equation}
\int_{\mathbb{R}^{N}}|\nabla v_{0}|^{2}dx+\int_{\mathbb{R}^{N}}\left\vert
x\right\vert ^{2}|v_{0}|^{2}dx-N\int_{\mathbb{R}^{N}}|v_{0}|^{2}dx\geq
C\left[  \int_{\mathbb{R}^{N}}|v_{0}|^{2}dx-\frac{\left(  \int_{\mathbb{R}%
^{N}}v_{0}e^{-\frac{1}{2}\left\vert x\right\vert ^{2}}dx\right)  ^{2}}%
{\int_{\mathbb{R}^{N}}e^{-\left\vert x\right\vert ^{2}}dx}\right]  \label{P0}%
\end{equation}
and
\begin{equation}
\int_{\mathbb{R}^{N+2k}}|\nabla v_{k}|^{2}dx+\int_{\mathbb{R}^{N+2k}%
}\left\vert x\right\vert ^{2}|v_{k}|^{2}dx-N\int_{\mathbb{R}^{N+2k}}%
|v_{k}|^{2}dx\geq C\int_{\mathbb{R}^{N+2k}}|v_{k}|^{2}dx\text{ }\forall
k\geq1\text{.} \label{Pk}%
\end{equation}
Note that we can rewrite (\ref{Pk}) as
\[
\int_{\mathbb{R}^{N+2k}}|\nabla v_{k}|^{2}dx+\int_{\mathbb{R}^{N+2k}%
}\left\vert x\right\vert ^{2}|v_{k}|^{2}dx\geq\left(  N+C\right)
\int_{\mathbb{R}^{N+2k}}|v_{k}|^{2}dx\text{ }\forall k\geq1\text{.}%
\]
Obviously, by the sharp HUP%
\[
\int_{\mathbb{R}^{N+2k}}|\nabla v_{k}|^{2}dx+\int_{\mathbb{R}^{N+2k}%
}\left\vert x\right\vert ^{2}|v_{k}|^{2}dx\geq\left(  N+2k\right)
\int_{\mathbb{R}^{N+2k}}|v_{k}|^{2}dx\text{ }\forall k\geq0\text{,}%
\]
we must have $C\leq2$. Now, we will show that for radial function $v\in
C_{0}^{\infty}(\mathbb{R}^{N}):$
\begin{equation}
\int_{\mathbb{R}^{N}}|\nabla v|^{2}dx+\int_{\mathbb{R}^{N}}\left\vert
x\right\vert ^{2}|v|^{2}dx-N\int_{\mathbb{R}^{N}}|v|^{2}dx\geq2\left[
\int_{\mathbb{R}^{N}}|v|^{2}dx-\frac{\left(  \int_{\mathbb{R}^{N}}%
ve^{-\frac{1}{2}\left\vert x\right\vert ^{2}}dx\right)  ^{2}}{\int
_{\mathbb{R}^{N}}e^{-\left\vert x\right\vert ^{2}}dx}\right]  . \label{P01}%
\end{equation}
Equivalently%
\[
\int_{\mathbb{R}^{N}}\left\vert \nabla\left(  ve^{\frac{1}{2}\left\vert
x\right\vert ^{2}}\right)  \right\vert e^{-\left\vert x\right\vert ^{2}}%
dx\geq2\left[  \int_{\mathbb{R}^{N}}\left\vert ve^{\frac{1}{2}\left\vert
x\right\vert ^{2}}\right\vert ^{2}e^{-\left\vert x\right\vert ^{2}}%
dx-\frac{\left(  \int_{\mathbb{R}^{N}}ve^{\frac{1}{2}\left\vert x\right\vert
^{2}}e^{-\left\vert x\right\vert ^{2}}dx\right)  ^{2}}{\int_{\mathbb{R}^{N}%
}e^{-\left\vert x\right\vert ^{2}}dx}\right]
\]
which is just the sharp Poincare inequality for radial functions.

In summary, we get that
\[
\theta_{3}\left(  u\right)  \geq2\left[  \int_{\mathbb{R}^{N}}\left\vert
u\right\vert ^{2}dx-\frac{\left(  \int_{\mathbb{R}^{N}}ue^{-\frac{1}%
{2}\left\vert x\right\vert ^{2}}dx\right)  ^{2}}{\int_{\mathbb{R}^{N}%
}e^{-\left\vert x\right\vert ^{2}}dx}\right]  \text{ }\forall u\in
C_{0}^{\infty}(\mathbb{R}^{N}).
\]

Now by applying the standard scaling argument, we obtain
\[
\left(  \int_{\mathbb{R}^{N}}\left\vert \nabla v\right\vert ^{2}dx\right)
^{\frac{1}{2}}\left(  \int_{\mathbb{R}^{N}}\left\vert v\right\vert
^{2}\left\vert x\right\vert ^{2}dx\right)  ^{\frac{1}{2}}-\frac{N}{2}%
\int_{\mathbb{R}^{N}}\left\vert v\right\vert ^{2}dx\geq d_{1}^{2}(v,E_{HUP}).
\]

\end{proof}

\subsection{ Proofs of Theorem \ref{T3} and Theorem \ref{T4}}

\begin{proof}
[Proof of Theorem \ref{T3}]Since the smooth field $\mathbf{U}$ is
conservative, we have $\mathbf{U}=\nabla u$ for some smooth function $u$. Then
$\delta_{3}(\mathbf{U})=\delta_{1}(u)$. Therefore
\[
\delta_{3}(\mathbf{U})=\delta_{1}(u)\geq\dfrac{C\left(  N\right)  }{2}%
\inf_{u^{\ast}\in E_{SHUP}}\left\Vert \nabla u\mathbf{-}\nabla u^{\ast
}\right\Vert _{2}^{2}=\dfrac{C\left(  N\right)  }{2}\inf_{\mathbf{U}^{\ast}\in
E_{CFHUP}}\left\Vert \mathbf{U-U}^{\ast}\right\Vert _{2}^{2}%
\]
and%
\begin{align*}
\delta_{3}(\mathbf{U})  &  =\delta_{1}(u)\geq\dfrac{C\left(  N\right)  }%
{4}\inf\left\{  \left\Vert \nabla u\mathbf{-}\nabla u^{\ast}\right\Vert
_{2}^{2}:u^{\ast}\in E_{SHUP}\text{ and }\left\Vert \nabla u\right\Vert
_{2}^{2}=\left\Vert \nabla u^{\ast}\right\Vert _{2}^{2}\right\} \\
&  =\dfrac{C\left(  N\right)  }{4}\inf\left\{  \left\Vert \mathbf{U-U}^{\ast
}\right\Vert _{2}^{2}:\mathbf{U}^{\ast}\in E_{CFHUP}\text{ and }\left\Vert
\mathbf{U}\right\Vert _{2}^{2}=\left\Vert \mathbf{U}^{\ast}\right\Vert
_{2}^{2}\right\}  .
\end{align*}

\end{proof}

\begin{proof}
[Proof of Theorem \ref{T4}]By Theorem \ref{T2}, we have
\[
\delta_{2}(u)\geq C\left(  N\right)  \inf_{c}\int_{\mathbb{R}^{N}}\left\vert
\nabla\left(  u-ce^{-\left\vert x\right\vert ^{2}/2}\right)  \right\vert
^{2}dx.
\]
Note that
\[
\delta_{2}(u)=\int_{\mathbb{R}^{N}}\left\Vert \nabla^{2}v-x\otimes\nabla
v\right\Vert _{HS}^{2}e^{-\left\vert x\right\vert ^{2}}dx
\]
with $v=ue^{\frac{\left\vert x\right\vert ^{2}}{2}\text{ }}$. Therefore%
\begin{align*}
&  \int_{\mathbb{R}^{N}}\left\Vert \nabla^{2}v-x\otimes\nabla v\right\Vert
_{HS}^{2}e^{-\left\vert x\right\vert ^{2}}dx\\
&  \geq C\left(  N\right)  \inf_{c}\int_{\mathbb{R}^{N}}\left\vert
\nabla\left(  u-ce^{-\left\vert x\right\vert ^{2}/2}\right)  \right\vert
^{2}dx\\
&  =C\left(  N\right)  \inf_{c}\int_{\mathbb{R}^{N}}\left\vert \nabla\left(
ve^{-\left\vert x\right\vert ^{2}/2}-ce^{-\left\vert x\right\vert ^{2}%
/2}\right)  \right\vert ^{2}dx\\
&  =C\left(  N\right)  \inf_{c}\int_{\mathbb{R}^{N}}\left\vert \nabla\left(
\left(  v-c\right)  e^{-\left\vert x\right\vert ^{2}/2}\right)  \right\vert
^{2}dx\\
&  =C\left(  N\right)  \inf_{c}\int_{\mathbb{R}^{N}}\left\vert \nabla\left(
v-c\right)  e^{-\left\vert x\right\vert ^{2}/2}+\left(  v-c\right)
\nabla\left(  e^{-\left\vert x\right\vert ^{2}/2}\right)  \right\vert ^{2}dx\\
&  =C\left(  N\right)  \inf_{c}\int_{\mathbb{R}^{N}}\left\vert \nabla
v-\left(  v-c\right)  x\right\vert ^{2}e^{-\left\vert x\right\vert ^{2}}dx.
\end{align*}

\end{proof}

\section{The stability of the second-order HUP-Proofs of Theorem \ref{L1},
Theorem \ref{T2}, Theorem \ref{T1} and Theorem \ref{T2.1}}

\begin{proof}
[Proof of Theorem \ref{L1}]$C\left(  N,0\right)  =0$ is a consequence of the
second order Heisenberg Uncertainty Principle on $\mathbb{R}^{N}$ proved in
\cite{CFL22}. Similarly, by second order Heisenberg Uncertainty Principle on
$\mathbb{R}^{N+2k}$, we have
\begin{align*}
C\left(  N,k\right)   &  =\inf_{u\text{ is radial}}\frac{\int_{\mathbb{R}%
^{N+2k}}|\Delta u|^{2}dx+\int_{\mathbb{R}^{N+2k}}\left[  |\nabla
u|^{2}\left\vert x\right\vert ^{2}-2k|u|^{2}\right]  dx}{\int_{\mathbb{R}%
^{N+2k}}|\nabla u|^{2}dx}-\left(  N+2\right) \\
&  \leq\inf_{u\text{ is radial}}\frac{\int_{\mathbb{R}^{N+2k}}|\Delta
u|^{2}dx+\int_{\mathbb{R}^{N+2k}}|\nabla u|^{2}\left\vert x\right\vert ^{2}%
dx}{\int_{\mathbb{R}^{N+2k}}|\nabla u|^{2}dx}-\left(  N+2\right) \\
&  =N+2k+2-\left(  N+2\right)  =2k.
\end{align*}
Now, we will show that $C\left(  N,k\right)  \geq\sqrt{\left(  N+2k\right)
^{2}-8k}-N$ $\forall k\geq1$. Therefore, $C\left(  N\right)  \geq\min_{k\geq
1}\left(  -N+\sqrt{\left(  N+2k\right)  ^{2}-8k}\right)  =K\left(  N\right)
$. Indeed, we have%
\begin{align*}
&  \int_{\mathbb{R}^{N+2k}}|\Delta u(\left\vert x\right\vert )|^{2}%
dx+\int_{\mathbb{R}^{N+2k}}\left[  |\nabla u(\left\vert x\right\vert
)|^{2}\left\vert x\right\vert ^{2}-2k|u(\left\vert x\right\vert )|^{2}\right]
dx\\
&  -(N+2+K)\int_{\mathbb{R}^{N+2k}}|\nabla u(\left\vert x\right\vert
)|^{2}dx\\
&  =|\mathbb{S}_{N+2k-1}|\int_{0}^{\infty}\left(  \left\vert u^{\prime\prime
}+\dfrac{N+2k-1}{r}u^{\prime}\right\vert ^{2}r^{N+2k-1}+\left\vert u^{\prime
}\right\vert ^{2}r^{N+2k+1}\right. \\
&  -\left.  (N+2+K)\left\vert u^{\prime}\right\vert ^{2}r^{N+2k-1}%
-2k\left\vert u\right\vert ^{2}r^{N+2k-1}\right)  dr\\
&  =|\mathbb{S}_{N+2k-1}|\int_{0}^{\infty}\left\vert u^{\prime\prime
}r+(N+2k-1)u^{\prime}+u^{\prime}r^{2}+\frac{2N+2k+K}{2}ur\right\vert
^{2}r^{N+2k-3}dr\\
&  +\left(  (N+2k)\frac{2N+2k+K}{2}-\left(  \frac{2N+2k+K}{2}\right)
^{2}-2k\right)  |\mathbb{S}_{N+2k-1}|\int_{0}^{\infty}\left\vert u\right\vert
^{2}r^{N+2k-1}dr.
\end{align*}
Therefore, we would like to choose $K>0$ be such that $\left(  \frac
{2N+2k+K}{2}\right)  ^{2}-(N+2k)\frac{2N+2k+K}{2}+2k\leq0$. Equivalently,
\[
K^{2}+2NK+\left(  4\left(  2-N\right)  k-4k^{2}\right)  \leq0\text{.}%
\]
That is $0<K\leq\sqrt{\left(  N+2k\right)  ^{2}-8k}-N$.

Next, by testing the constant $C\left( N,1\right)$ with the radial function $U(x)=e^{-\frac{|x|^2}{2}}$, we deduce that 
$$C\left( N,1\right) \leq \frac{2N}{N+2}<\sqrt{(N+2k)^2-8k}-N \leq C\left( N,k\right) ~\forall k\geq 2.$$ 
Therefore $C(N)=C(N,1)$.
\end{proof}

\begin{proof}
[Proof of Theorem \ref{T2}]Our goal is to determine the constant $K>0$ such
that%
\begin{equation}
\int_{\mathbb{R}^{N}}|\Delta u|^{2}dx+\int_{\mathbb{R}^{N}}\left\vert
x\right\vert ^{2}|\nabla u|^{2}dx-(N+2)\int_{\mathbb{R}^{N}}|\nabla
u|^{2}dx\geq K\inf_{c}\int_{\mathbb{R}^{N}}\left\vert \nabla\left(
u-ce^{-\left\vert x\right\vert ^{2}/2}\right)  \right\vert ^{2}dx.
\label{2ndOrdStab}%
\end{equation}
From the computations in the previous section, we have the following
identity:
\begin{align*}
&  \int_{\mathbb{R}^{N}}|\Delta u|^{2}dx+\int_{\mathbb{R}^{N}}\left\vert
x\right\vert ^{2}|\nabla u|^{2}dx-(N+2)\int_{\mathbb{R}^{N}}|\nabla u|^{2}dx\\
&  =\sum_{k=0}^{\infty}\left(  \int_{\mathbb{R}^{N+2k}}|\Delta v_{k}%
(\left\vert x\right\vert )|^{2}dx+\int_{\mathbb{R}^{N+2k}}\left[  |\nabla
v_{k}(\left\vert x\right\vert )|^{2}\left\vert x\right\vert ^{2}%
-2k|v_{k}(\left\vert x\right\vert )|^{2}\right]  dx\right. \\
&  \left.  -(N+2)\int_{\mathbb{R}^{N+2k}}|\nabla v_{k}(\left\vert x\right\vert
)|^{2}dx\right)  ,
\end{align*}
where $u(x)=u(r\sigma)=\sum_{k=0}^{\infty}v_{k}(r)r^{k}\phi_{k}(\sigma
).$\newline Moreover, we also have
\begin{align*}
\inf_{c}\int_{\mathbb{R}^{N}}\left\vert \nabla\left(  u-ce^{-\left\vert
x\right\vert ^{2}/2}\right)  \right\vert ^{2}dx  &  =\int_{\mathbb{R}^{N}%
}\left\vert \nabla u-\left(  \dfrac{\int_{\mathbb{R}^{N}}\nabla u\cdot
\nabla\left(  e^{-\left\vert x\right\vert ^{2}/2}\right)  dx}{\int
_{\mathbb{R}^{N}}\left\vert \nabla\left(  e^{-\left\vert x\right\vert ^{2}%
/2}\right)  \right\vert ^{2}dx}\right)  \nabla\left(  e^{-\left\vert
x\right\vert ^{2}/2}\right)  \right\vert ^{2}dx\\
&  =\int_{\mathbb{R}^{N}}\left\vert \nabla u\right\vert ^{2}dx-\dfrac{\left(
\int_{\mathbb{R}^{N}}\nabla u\cdot\nabla\left(  e^{-\left\vert x\right\vert
^{2}/2}\right)  dx\right)  ^{2}}{\int_{\mathbb{R}^{N}}\left\vert \nabla\left(
e^{-\left\vert x\right\vert ^{2}/2}\right)  \right\vert ^{2}dx}.
\end{align*}
We note that
\begin{align*}
\int_{\mathbb{R}^{N}}\nabla u\cdot\nabla\left(  e^{-\left\vert x\right\vert
^{2}/2}\right)  dx  &  =-\int_{\mathbb{R}^{N}}u\Delta\left(  e^{-\left\vert
x\right\vert ^{2}/2}\right)  dx\\
&  =-\int_{\mathbb{R}^{N}}\left(  \sum_{k=0}^{\infty}v_{k}(\left\vert
x\right\vert )\left\vert x\right\vert ^{k}\phi_{k}\left(  \frac{x}{\left\vert
x\right\vert }\right)  \right)  \Delta\left(  e^{-\left\vert x\right\vert
^{2}/2}\right)  dx\\
&  =-\sum_{k=0}^{\infty}\int_{\mathbb{R}^{N}}v_{k}(\left\vert x\right\vert
)\left\vert x\right\vert ^{k}\phi_{k}\left(  \frac{x}{\left\vert x\right\vert
}\right)  \Delta\left(  e^{-\left\vert x\right\vert ^{2}/2}\right)  dx\\
&  =-\int_{\mathbb{R}^{N}}v_{0}(\left\vert x\right\vert )\phi_{0}\left(
\frac{x}{\left\vert x\right\vert }\right)  \Delta\left(  e^{-\left\vert
x\right\vert ^{2}/2}\right)  dx\\
&  =\int_{\mathbb{R}^{N}}\nabla v_{0}(\left\vert x\right\vert )\cdot
\nabla\left(  e^{-\left\vert x\right\vert ^{2}/2}\right)  dx,
\end{align*}
since $\displaystyle\int_{\mathbb{S}^{N-1}}\phi_{k}(\sigma)d\sigma=0$ for all
$k\geq1$, and $\phi_{0}(\sigma)\equiv1$. Therefore, \eqref{2ndOrdStab} is
equivalent to
\begin{align*}
&  \sum_{k=0}^{\infty}\left(  \int_{\mathbb{R}^{N+2k}}|\Delta v_{k}(\left\vert
x\right\vert )|^{2}dx+\int_{\mathbb{R}^{N+2k}}\left[  |\nabla v_{k}(\left\vert
x\right\vert )|^{2}\left\vert x\right\vert ^{2}-2k|v_{k}(\left\vert
x\right\vert )|^{2}\right]  dx\right. \\
&  \left.  -(N+2+K)\int_{\mathbb{R}^{N+2k}}|\nabla v_{k}(\left\vert
x\right\vert )|^{2}dx\right)  +K\dfrac{\left(  \int_{\mathbb{R}^{N}}\nabla
v_{0}(\left\vert x\right\vert )\cdot\nabla\left(  e^{-\left\vert x\right\vert
^{2}/2}\right)  dx\right)  ^{2}}{\int_{\mathbb{R}^{N}}\left\vert \nabla\left(
e^{-\left\vert x\right\vert ^{2}/2}\right)  \right\vert ^{2}dx}\geq0.
\end{align*}
Equivalently, we need to find the largest value of $K$ such that%
\begin{align}
\int_{\mathbb{R}^{N}}|\Delta v_{0}(\left\vert x\right\vert )|^{2}%
dx+\int_{\mathbb{R}^{N}}|\nabla v_{0}(\left\vert x\right\vert )|^{2}\left\vert
x\right\vert ^{2}dx-(N+2+K)\int_{\mathbb{R}^{N}}|\nabla v_{0}(\left\vert
x\right\vert )|^{2}dx  & \nonumber\\
+K\dfrac{\left(  \int_{\mathbb{R}^{N}}\nabla v_{0}(\left\vert x\right\vert
)\cdot\nabla\left(  e^{-\left\vert x\right\vert ^{2}/2}\right)  dx\right)
^{2}}{\int_{\mathbb{R}^{N}}\left\vert \nabla\left(  e^{-\left\vert
x\right\vert ^{2}/2}\right)  \right\vert ^{2}dx}  &  \geq0 \label{N1}%
\end{align}
and%
\begin{align}
&  \int_{\mathbb{R}^{N+2k}}|\Delta v_{k}(\left\vert x\right\vert )|^{2}%
dx+\int_{\mathbb{R}^{N+2k}}\left[  |\nabla v_{k}(\left\vert x\right\vert
)|^{2}\left\vert x\right\vert ^{2}-2k|v_{k}(\left\vert x\right\vert
)|^{2}\right]  dx\label{N2}\\
&  -(N+2+K)\int_{\mathbb{R}^{N+2k}}|\nabla v_{k}(\left\vert x\right\vert
)|^{2}dx\geq0\text{ }\forall k\geq1\text{.}\nonumber
\end{align}
To deal with (\ref{N1}), we set $v(\left\vert x\right\vert )=\frac
{v_{0}^{\prime}(\left\vert x\right\vert )}{\left\vert x\right\vert }$. Then we
have that from Lemma \ref{T3.1}, \eqref{N1} is equivalent to
\begin{align*}
&  \int_{\mathbb{R}^{N+2}}|\nabla v(\left\vert x\right\vert )|^{2}%
dx+\int_{\mathbb{R}^{N+2}}\left\vert x\right\vert ^{2}|v(\left\vert
x\right\vert )|^{2}dx-(N+2)\int_{\mathbb{R}^{N+2}}|v(\left\vert x\right\vert
)|^{2}dx\\
&  \geq K\left(  \int_{\mathbb{R}^{N+2}}|v(\left\vert x\right\vert
)|^{2}dx-\dfrac{\left(  \int_{\mathbb{R}^{N+2}}v(\left\vert x\right\vert
)e^{-\left\vert x\right\vert ^{2}/2}dx\right)  ^{2}}{\int_{\mathbb{R}^{N+2}%
}e^{-\left\vert x\right\vert ^{2}}dx}\right)  .
\end{align*}
Moreover,
\[
\int_{\mathbb{R}^{N+2}}|v(\left\vert x\right\vert )|^{2}dx-\dfrac{\left(
\int_{\mathbb{R}^{N+2}}v(\left\vert x\right\vert )e^{-\left\vert x\right\vert
^{2}/2}dx\right)  ^{2}}{\int_{\mathbb{R}^{N+2}}e^{-\left\vert x\right\vert
^{2}}dx}=\inf_{c}\int_{\mathbb{R}^{N+2}}\left\vert v-ce^{-\left\vert
x\right\vert ^{2}/2}\right\vert ^{2}dx.
\]
Therefore, by Theorem \ref{A},
\[
0<K\leq2.
\]
Also, (\ref{N2}) is equivalent to
\[
K\leq C\left(  N,k\right)  .
\]
Therefore
\[
\delta_{2}(u)\geq K\inf_{u^{\ast}\in E_{SHUP}}\left\Vert \nabla\left(
u-u^{\ast}\right)  \right\Vert _{2}^{2}%
\]
with
\[
K=\min_{k\geq1}\left\{  C\left(  N,k\right)  ,2\right\}  =C\left(  N\right)
\geq\dfrac{4N-4}{\sqrt{N^{2}+4N-4}+N}.
\]

Now, we will show that $K=C\left(  N\right)  $ is optimal in
\eqref{2ndOrdStab}. Indeed, assume by contradiction that there exists
$K>C\left(  N\right)  $ such that
\[
\int_{\mathbb{R}^{N}}|\Delta u|^{2}dx+\int_{\mathbb{R}^{N}}\left\vert
x\right\vert ^{2}|\nabla u|^{2}dx-(N+2)\int_{\mathbb{R}^{N}}|\nabla
u|^{2}dx\geq K\inf_{c}\int_{\mathbb{R}^{N}}\left\vert \nabla\left(
u-ce^{-\left\vert x\right\vert ^{2}/2}\right)  \right\vert ^{2}dx.
\]
Let $\varepsilon=\frac{K-C\left(  N\right)  }{2}>0$. Then we can find
$m\in\mathbb{N}$ such that
\begin{align*}
C\left(  N,m\right)   &  =\inf_{u\text{ is radial}}\frac{\int_{\mathbb{R}%
^{N+2m}}|\Delta u|^{2}dx+\int_{\mathbb{R}^{N+2m}}\left[  |\nabla
u|^{2}\left\vert x\right\vert ^{2}-2m|u|^{2}\right]  dx}{\int_{\mathbb{R}%
^{N+2m}}|\nabla u|^{2}dx}-\left(  N+2\right) \\
&  <C\left(  N\right)  +\varepsilon\text{.}%
\end{align*}
Moreover, there exists a sequence $v_{j}(r)$ such that $\int_{\mathbb{R}%
^{N+2m}}|\nabla v_{j}|^{2}dx=1$ and
\[
\int_{\mathbb{R}^{N+2m}}|\Delta v_{j}|^{2}dx+\int_{\mathbb{R}^{N+2m}}\left[
|\nabla v_{j}|^{2}\left\vert x\right\vert ^{2}-2m|v_{j}|^{2}\right]
dx-\left(  N+2\right)  \downarrow C\left(  N,m\right)  \text{ as }%
j\rightarrow\infty\text{.}%
\]
Now, if we choose $u_{j}=v_{j}(r)r^{m}\phi_{m}(\sigma)$, then by the above
computations, we have that as $j\rightarrow\infty:$%
\begin{align*}
&  \int_{\mathbb{R}^{N}}|\Delta u_{j}|^{2}dx+\int_{\mathbb{R}^{N}}\left\vert
x\right\vert ^{2}|\nabla u_{j}|^{2}dx-(N+2)\int_{\mathbb{R}^{N}}|\nabla
u_{j}|^{2}dx-K\inf_{c}\int_{\mathbb{R}^{N}}\left\vert \nabla\left(
u_{j}-ce^{-\left\vert x\right\vert ^{2}/2}\right)  \right\vert ^{2}dx\\
&  =\int_{\mathbb{R}^{N+2m}}|\Delta v_{j}|^{2}dx+\int_{\mathbb{R}^{N+2m}%
}\left[  |\nabla v_{j}|^{2}\left\vert x\right\vert ^{2}-2m|v_{j}|^{2}\right]
dx-\left(  N+2+K\right) \\
&  \rightarrow C\left(  N,m\right)  -K<C\left(  N\right)  +\varepsilon-K<0
\end{align*}

which is a contradiction.
\end{proof}

\begin{proof}
[Proof of Theorem \ref{T1}]Using $u_{\lambda}(x)=u(\lambda x)$ instead of $u$
in \eqref{2ndOrdStab}, we get
\begin{equation}
\int_{\mathbb{R}^{N}}|\Delta u_{\lambda}|^{2}dx+\int_{\mathbb{R}^{N}%
}\left\vert x\right\vert ^{2}|\nabla u_{\lambda}|^{2}dx-(N+2)\int
_{\mathbb{R}^{N}}|\nabla u_{\lambda}|^{2}dx\geq C\left(  N\right)  \inf
_{c}\int_{\mathbb{R}^{N}}\left\vert \nabla\left(  u_{\lambda}-ce^{-\left\vert
x\right\vert ^{2}/2}\right)  \right\vert ^{2}dx, \label{scal2ndOrdStab}%
\end{equation}
which is equivalent to
\begin{align*}
\lambda^{4-N}\int_{\mathbb{R}^{N}}|\Delta u|^{2}dx+\lambda^{-N}\int
_{\mathbb{R}^{N}}\left\vert x\right\vert ^{2}|\nabla u|^{2}dx  &
-(N+2)\lambda^{2-N}\int_{\mathbb{R}^{N}}|\nabla u|^{2}dx\\
&  \geq\lambda^{2-N}C\left(  N\right)  \inf_{c}\int_{\mathbb{R}^{N}}\left\vert
\nabla\left(  u-ce^{-\left\vert x\right\vert ^{2}/(2\lambda^{2})}\right)
\right\vert ^{2}dx.
\end{align*}
Then,
\begin{align*}
\lambda^{2}\int_{\mathbb{R}^{N}}|\Delta u|^{2}dx+\lambda^{-2}\int
_{\mathbb{R}^{N}}\left\vert x\right\vert ^{2}|\nabla u|^{2}dx  &
-(N+2)\int_{\mathbb{R}^{N}}|\nabla u|^{2}dx\\
&  \geq C\left(  N\right)  \inf_{c}\int_{\mathbb{R}^{N}}\left\vert
\nabla\left(  u-ce^{-\left\vert x\right\vert ^{2}/(2\lambda^{2})}\right)
\right\vert ^{2}dx.
\end{align*}
By choosing $\lambda=\left(  \dfrac{\int_{\mathbb{R}^{N}}\left\vert
x\right\vert ^{2}|\nabla u|^{2}dx}{\int_{\mathbb{R}^{N}}|\Delta u|^{2}%
dx}\right)  ^{1/4}$, we obtain
\begin{align*}
2\left(  \int_{\mathbb{R}^{N}}|\Delta u|^{2}dx\right)  ^{1/2}\left(
\int_{\mathbb{R}^{N}}\left\vert x\right\vert ^{2}|\nabla u|^{2}dx\right)
^{1/2}  &  -(N+2)\int_{\mathbb{R}^{N}}|\nabla u|^{2}dx\\
&  \geq C\left(  N\right)  \inf_{c}\int_{\mathbb{R}^{N}}\left\vert
\nabla\left(  u-ce^{-\left\vert x\right\vert ^{2}/(2\lambda^{2})}\right)
\right\vert ^{2}dx.
\end{align*}
Therefore,
\begin{align*}
\left(  \int_{\mathbb{R}^{N}}|\Delta u|^{2}dx\right)  ^{1/2}\left(
\int_{\mathbb{R}^{N}}\left\vert x\right\vert ^{2}|\nabla u|^{2}dx\right)
^{1/2}  &  -\dfrac{N+2}{2}\int_{\mathbb{R}^{N}}|\nabla u|^{2}dx\\
&  \geq\dfrac{C\left(  N\right)  }{2}\inf_{c}\int_{\mathbb{R}^{N}}\left\vert
\nabla\left(  u-ce^{-\left\vert x\right\vert ^{2}/(2\lambda^{2})}\right)
\right\vert ^{2}dx\\
&  \geq\dfrac{C\left(  N\right)  }{2}\inf_{u^{\ast}\in E_{SHUP}}\left\Vert
\nabla\left(  u-u^{\ast}\right)  \right\Vert _{2}^{2}.
\end{align*}
Since $C\left(  N\right)  $ is sharp in Theorem \ref{T2}, it is easy to see
that $\dfrac{C\left(  N\right)  }{2}$ is sharp.
\end{proof}

\begin{proof}
[Proof of Theorem \ref{T2.1}]We will use Theorem \ref{T1} to show that
\begin{equation}
\delta_{1}(u)\geq\dfrac{C\left(  N\right)  }{4}\inf\left\{  \left\Vert
\nabla\left(  u-u^{\ast}\right)  \right\Vert _{2}^{2}:u^{\ast}\in
E_{SHUP}\text{ and }\left\Vert \nabla u\right\Vert _{2}^{2}=\left\Vert \nabla
u^{\ast}\right\Vert _{2}^{2}\right\}  . \label{cdtStab}%
\end{equation}
Indeed, first of all, we can assume without loss of generality that
$\left\Vert \nabla u\right\Vert _{2}^{2}=1$ since there is nothing to do if
$\left\Vert \nabla u\right\Vert _{2}^{2}=0$. We will split into two cases:

\textbf{Case }$1$\textbf{:} $\delta_{1}(u)<\frac{C\left(  N\right)  }{2}$. In
this case, we first claim that there exists a function $v\in E_{SHUP}$ such
that
\[
\displaystyle\inf_{u^{\ast}\in E_{SHUP}}\Vert\nabla(u-u^{\ast})\Vert_{2}%
^{2}=\Vert\nabla(u-v)\Vert_{2}^{2}.
\]
Indeed, we have
\[
\displaystyle\inf_{u^{\ast}\in E_{SHUP}}\Vert\nabla(u-u^{\ast})\Vert_{2}%
^{2}\leq\dfrac{2}{C\left(  N\right)  }\delta_{1}(u)<1.
\]
Let $\{v_{j}\}_{j}$ be a sequence in $E_{SHUP}$ such that
\[
\lim_{j\rightarrow\infty}\Vert\nabla u-\nabla v_{j}\Vert_{2}^{2}%
=\displaystyle\inf_{u^{\ast}\in E}\Vert\nabla(u-u^{\ast})\Vert_{2}^{2}<1.
\]
Using the triangle inequality, we obtain
\[
\Vert\nabla(u-v_{j})\Vert_{2}\geq\left\vert \Vert\nabla u\Vert_{2}-\Vert
\nabla(v_{j})\Vert_{2}\right\vert =\left\vert 1-\Vert\nabla(v_{j})\Vert
_{2}\right\vert ,
\]
which implies that
\[
\lim_{j\rightarrow\infty}\left\vert 1-\Vert\nabla(v_{j})\Vert_{2}\right\vert
<1
\]
or there exist positive constants $C,c>0$ such that
\[
0<c\leq\Vert\nabla(v_{j})\Vert_{L^{2}(\mathbb{R}^{N})}\leq C<2,
\]
for $j$ large enough. With $v_{j}=\alpha_{j}e^{-\frac{\beta_{j}}{2}|x|^{2}}$,
as above, we get
\[
\int_{\mathbb{R}^{N}}|\nabla v_{j}|^{2}dx=\alpha_{j}^{2}\beta_{j}^{2}%
\int_{\mathbb{R}^{N}}e^{-\beta_{j}|x|^{2}}|x|^{2}dx=\alpha_{j}^{2}\beta
_{j}^{1-\frac{N}{2}}\int_{\mathbb{R}^{N}}e^{-|x|^{2}}|x|^{2}dx=\dfrac{N}{2}%
\pi^{N/2}\alpha_{j}^{2}\beta_{j}^{1-\frac{N}{2}},
\]
which allows us there exist constants $M,m>0$ such that
\[
m^{2}\leq\alpha_{j}^{2}\beta_{j}^{1-\frac{N}{2}}\leq M^{2},
\]
for $j$ large enough. Using the same arguments as the ones in Lemma
\ref{PropOfE}, up to a subsequence, we can assume that $\alpha_{j}%
\rightarrow\alpha\in\mathbb{R}$, and $\beta_{j}\rightarrow\beta\in(0,\infty)$.
Let us define $v(x)=\alpha e^{-\frac{\beta}{2}|x|^{2}}\in E_{SHUP}$, we have
$\nabla v_{j}\rightarrow\nabla v$ in $L^{2}(\mathbb{R}^{N})$. Hence, using the
triangle inequality again, we get
\[
\Vert\nabla(u-v)\Vert_{2}^{2}=\lim_{j\rightarrow\infty}\Vert\nabla
(u-v_{j})\Vert_{2}^{2}=\inf_{u^{\ast}\in E}\Vert\nabla(u-u^{\ast})\Vert
_{2}^{2},
\]
which proves our claim.

From the claim, we have
\[
\int_{\mathbb{R}^{N}}|\nabla(u-v)|^{2}dx\leq\dfrac{2}{C\left(  N\right)
}\delta_{1}(u)<1.
\]
Therefore, since $\Vert\nabla u\Vert_{2}^{2}=1$, $\Vert\nabla v\Vert_{2}%
^{2}\neq0$ (otherwise, $v$ would be constant, which is a contradiction). Thus,
we can define $\lambda=\left(  \int_{\mathbb{R}^{N}}|\nabla v|^{2}\right)
^{-1/2}>0$, and a function $w=\lambda v\in E_{SHUP}$ satisfying $\Vert\nabla
w\Vert_{2}^{2}=\lambda^{2}\Vert\nabla v\Vert_{2}^{2}=1.$ Hence,
\begin{align*}
\dfrac{C\left(  N\right)  }{2}\inf\{\Vert\nabla(u-u^{\ast})\Vert_{2}^{2}  &
:u^{\ast}\in E_{SHUP},\Vert\nabla u^{\ast}\Vert_{2}^{2}=\Vert\nabla u\Vert
_{2}^{2}\}\\
&  \leq\dfrac{C\left(  N\right)  }{2}\int_{\mathbb{R}^{N}}|\nabla
(u-w)|^{2}dx\\
&  =\dfrac{C\left(  N\right)  }{2}\left(  2-2\lambda\int_{\mathbb{R}^{N}%
}\nabla u\cdot\nabla vdx\right)  .
\end{align*}
Moreover,
\[
\int_{\mathbb{R}^{N}}|\nabla(u-v)|^{2}dx=1-2\int_{\mathbb{R}^{N}}\nabla
u\cdot\nabla vdx+\dfrac{1}{\lambda^{2}},
\]
which implies that
\[
1-2\int_{\mathbb{R}^{N}}\nabla u\cdot\nabla vdx+\dfrac{1}{\lambda^{2}}%
\leq\dfrac{1}{K\left(  N\right)  }\delta_{2}(u)<1,
\]
and
\[
0<\dfrac{1}{2\lambda^{2}}\leq\int_{\mathbb{R}^{N}}\nabla u\cdot\nabla vdx.
\]
Then, to prove \eqref{cdtStab}, it suffices to prove that
\begin{align*}
2\delta_{1}(u)  &  \geq C\left(  N\right)  \left(  1-2\int_{\mathbb{R}^{N}%
}\nabla u\cdot\nabla vdx+\dfrac{1}{\lambda^{2}}\right) \\
&  \geq\dfrac{C\left(  N\right)  }{2}\left(  2-2\lambda\int_{\mathbb{R}^{N}%
}\nabla u\cdot\nabla vdx\right)  ,
\end{align*}
which is equivalent to
\[
(2-\lambda)\int_{\mathbb{R}^{N}}\nabla u\cdot\nabla vdx\leq\dfrac{1}%
{\lambda^{2}}.
\]
To do this, by using H\"{o}lder's inequality, we get
\[
\lambda\int_{\mathbb{R}^{N}}\nabla u\cdot\nabla vdx=\int_{\mathbb{R}^{N}%
}\nabla u\cdot\nabla wdx\leq1.
\]
That is
\[
\dfrac{1}{2\lambda^{2}}\leq\int_{\mathbb{R}^{N}}\nabla u\cdot\nabla
vdx\leq\dfrac{1}{\lambda}.
\]
Therefore,
\[
(2-\lambda)\int_{\mathbb{R}^{N}}\nabla u\cdot\nabla vdx\leq(2-\lambda)\frac
{1}{\lambda}\leq\dfrac{1}{\lambda^{2}},
\]
with $0<\lambda\leq2$. Otherwise, if $\lambda>2$, then
\[
(2-\lambda)\int_{\mathbb{R}^{N}}\nabla u\cdot\nabla vdx<0<\dfrac{1}%
{\lambda^{2}}.
\]

\textbf{Case }$2$: $\delta_{1}(u)\geq\frac{C\left(  N\right)  }{2}$. This is
the easier case since we have
\begin{align*}
&  \inf\{\Vert\nabla(u-u^{\ast})\Vert_{2}^{2}:u^{\ast}\in E_{SHUP},\Vert\nabla
u^{\ast}\Vert_{2}^{2}=\Vert\nabla u\Vert_{2}^{2}=1\}\\
&  =\inf\{\Vert\nabla(u+u^{\ast})\Vert_{2}^{2}:u^{\ast}\in E_{SHUP}%
,\Vert\nabla u^{\ast}\Vert_{2}^{2}=\Vert\nabla u\Vert_{2}^{2}=1\}\\
&  \leq\dfrac{1}{2}\inf\{\Vert\nabla(u+u^{\ast})\Vert_{2}^{2}+\Vert
\nabla(u-u^{\ast})\Vert_{2}^{2}:u^{\ast}\in E_{SHUP},\Vert\nabla u^{\ast}%
\Vert_{2}^{2}=\Vert\nabla u\Vert_{2}^{2}=1\}\\
&  =\dfrac{1}{2}\inf\{2\left(  \Vert\nabla(u)\Vert_{2}^{2}+\Vert\nabla
(u^{\ast})\Vert_{2}^{2}\right)  :u^{\ast}\in E_{SHUP},\Vert\nabla u^{\ast
}\Vert_{2}^{2}=\Vert\nabla u\Vert_{2}^{2}=1\}\\
&  =2\leq\dfrac{4}{C\left(  N\right)  }\delta_{1}(u),
\end{align*}
which implies \eqref{cdtStab}.
\end{proof}

\section*{Acknowledgments}

The authors would like to sincerely thank the referee for their careful reading and insightful comments. In particular, the referee pointed out that the best constant \(C(N) = \inf_k C(N,k)\) in our original work is actually attained at \(k=1\), i.e., \(C(N) = C(N,1)\). This valuable observation significantly clarified the result and strengthened the paper. A. Do and G. Lu were partially supported by grants
from the Simons Foundation and a Simons Fellowship. N. Lam was partially
supported by an NSERC Discovery Grant.





\end{document}